\documentclass{amsart}
\usepackage{amsmath,amssymb}

\usepackage{mathtools}
\usepackage{booktabs} % For better table lines
\usepackage{adjustbox} % To constrain height
\usepackage{color,soul}
\usepackage{bbm}
\usepackage[linesnumbered,ruled,vlined]{algorithm2e}
\SetKwInput{KwData}{Given data}                % Set the Input
\SetKwInput{KwInput}{Input}                % Set the Input
\SetKwInput{KwOutput}{Output}              % set the Output
\usepackage{hyperref}
\hypersetup{colorlinks=true, urlcolor=blue, citecolor=blue, linkcolor=blue}

\usepackage{amsopn,amsthm}
\usepackage{doi}

\newtheorem{theorem}{Theorem}[section]
\newtheorem{lemma}[theorem]{Lemma}
\newtheorem{corollary}[theorem]{Corollary}

\theoremstyle{definition}

\theoremstyle{remark}
\newtheorem{remark}[theorem]{Remark}
\newcommand{\dif}{\mathop{}\!\mathrm{d}}  
\numberwithin{equation}{section}
\def\p{\partial}
\def\N{\mathbb N}
\def\R{\mathbb R}

\newcommand{\dx}{\, \dif x}
\newcommand{\ds}{\, \dif s}

\makeatletter
\@namedef{subjclassname@2020}{\textup{2020} Mathematics Subject Classification}
\makeatother

\begin{document}

\title[{A Computational Method for the Inverse Robin Problem}]{A Computational Method for the Inverse Robin Problem with Convergence Rate}

\author{Erik Burman}
\address{Department of Mathematics, University College London, 807b Gower Street, London, WC1E
6BT, United Kingdom}
\curraddr{}
\email{e.burman@ucl.ac.uk}
\thanks{E.B. was supported by the EPSRC grants EP/T033126/1 and
  EP/V050400/1. For the purpose of open access, the author has applied a Creative Commons Attribution (CC BY) licence to any Author Accepted Manuscript version arising.}

\author{Marvin Kn\"oller}
\address{Department of Mathematics and Statistics, University of Helsinki, P.O 68, 00014, Helsinki,
Finland}
\curraddr{}
\email{marvin.knoller@helsinki.fi}
\thanks{M.K. was supported by the Research Council of Finland (Flagship of Advanced
Mathematics for Sensing, Imaging and Modelling grant 359182).}

\author{Lauri Oksanen}
\address{Department of Mathematics and Statistics, University of Helsinki, P.O 68, 00014, Helsinki,
Finland}
\curraddr{}
\email{lauri.oksanen@helsinki.fi}
\thanks{L.O. was supported by the European
Research Council of the European Union, grant 101086697 (LoCal), and the Reseach Council of
Finland, grants 347715, 353096 and 359182. Views and opinions expressed are those of the authors only and do not necessarily reflect
those of the European Union or the other funding organizations.}

\subjclass[2020]{
65N21, 
65N12
}
\date{}

\begin{abstract}

The inverse Robin problem consists in determining the Robin parameter in an elliptic partial differential equation posed on a domain $\Omega$.
In this work it is assumed that the right-hand sides of the elliptic problem are known, that the associated solution is available on a subdomain $\omega \subset \Omega$ 
and that this solution vanishes at most at isolated points on the boundary $\partial \Omega$.
Within this framework, a computational method for the reconstruction of the Robin parameter is established, which is inspired by a unique continuation method. 
The proposed scheme relies solely on first-order Lagrange finite elements ensuring a straightforward implementation.
Under the main assumption that the Robin parameter is in a finite dimensional space of continuously differentiable functions
it is shown that the numerical method is second order convergent in the finite element's mesh size.
For noisy data this convergence rate is shown to hold true until the noise term dominates the error estimate.
Numerical experiments are presented that highlight the feasibility of the Robin parameter reconstruction and that confirm the theoretical convergence results numerically.

\end{abstract}
\maketitle

\section{Introduction}
The unique continuation principle asserts that the knowledge of an elliptic partial differentiable equation's solution on an open set $\omega \subset \Omega$ is enough to extend it to its entire underlying domain $\Omega$ uniquely.
When linear Robin boundary conditions are imposed on $\partial \Omega$, the Robin parameter can additionally be recovered theoretically under the assumption that the elliptic problem's solution does not vanish on any open subset of the boundary and that the Robin parameter is continuous.
In practice, however, the reconstruction of the Robin parameter is more delicate. 
The introduction of finite dimensional function spaces breaks down arguments about the identification of the Robin parameter, which previously were valid on a continuum level.
The purpose of this work is to introduce and analyze a Newton method for recovering the Robin parameter numerically based on first order Lagrange finite elements only.
Our main result shows that the proposed method converges quadratically in the finite element's mesh size to the true Robin parameter under the assumption that the Robin parameter lies in a known finite dimensional space of continuously differentiable functions and that the corresponding boundary problem's solution vanishes at most in isolated points on $\partial \Omega$. 
Furthermore we show that when the given data on the subdomain $\omega$ is perturbed by some noise, then the quadratic convergence holds true until the noise term dominates the error bound.

The identification of the Robin coefficient in a second order elliptic problem using measured data is an important problem in many applications such as corrosion detection \cite{Kaup1995}.
Early stability estimates of logarithmic type were proposed in \cite{CFJL04, Chou04} and, for finite dimensional spaces, under the assumption that the Robin coefficient can be expressed as a piecewise constant function, in \cite{Sin07}. The inverse problem for the Robin coefficient is intricately linked to the unique continuation problem, which was used for its solution in \cite{Inglese1997}. The quasi reversibility method \cite{LattesLions1969} was applied to the reconstruction of the Robin coefficient in \cite{FI99}. More generally the problem is typically recast as an optimization problem or a least squares problem \cite{CJ99}.  
The regularization of such systems and their solution using conjugate gradient methods was investigated in \cite{Jin07,JZ09,JZ10}.  
Another approach is to apply the Kohn-Vogelius penalty method \cite{KohnVogelius1984}. It was first applied in this context in \cite{CEJ04} and more recently, providing error estimates in the finite dimensional case, in \cite{BurCenJinZhou25}. The questions of uniqueness and stability and global convergence have been studied in \cite{Har21,Har19}. All approximation methods based on control or least squares suffer from the potential existence of local minima. Recently, in \cite{Har22}, an approach to avoid this problem was proposed by recasting the problem into a convex
              non-linear semidefinite programming problem. Finally we can point out that recently also Bayesian methods have been applied for the solution of the identification problem \cite{RSGK24}.
                  
Computational unique continuation was studied for Poisson's problem with Dirichlet and Neumann conditions in \cite{BurOks24} and \cite{BurOksZhi25}, respectively.
In both works the assumption that the respective trace of the solution lies in a finite dimensional space plays an essential role.
By using tailored stability results and by studying variational formulations augmented by some additional stabilization terms, the respective unique continuation problems are approached numerically.

In this work we revisit the control formulation for the reconstruction of a smooth Robin coefficient lying in a finite dimensional space. Drawing on ideas for the finite element approximation of unique continuation problems with finite dimensional trace \cite{BurOks24} we derive a new Lipschitz stability estimate with additional Robin-type constraint. 
Then this estimate is used in the framework of \cite{keller75} to prove a priori error estimates for discrete solutions and second order convergence of the Newton method for initial guesses close enough to the discrete solution.

This work is structured as follows.
In the second section we motivate the inverse Robin problem and explain our methodological assumptions in detail.
In the third section we introduce the Robin problem on polygonal domains and recall well-posedness results. 
We formulate the inverse Robin problem and construct a suitable Lagrangian associated to it. 
The identification of the Lagrangian's saddle points ultimately leads to the task to determine the root of a single function.
Finding this root is our approach in reconstructing the Robin parameter.
In the fourth section we recall finite element error bounds and study linearizations of finite element solutions to Robin type problems with respect to the Robin parameter.
The fifth section combines our finite element bounds with the work \cite{keller75} to derive convergence rates and stability of the Newton algorithm.
Section six is about the perturbation analysis, where the measured data is a noisy version of a Robin boundary problem solution.
Finally, in section seven, numerical examples are presented. These examples highlight the efficacy of the proposed method and confirm the convergence rates numerically.

\section{The inverse Robin problem}
In this short section we motivate the inverse Robin problem. We postpone the introduction of function spaces and regularity requirements to the next section.

Denote by $\Omega$ a sufficiently regular domain and
let $f$ and $g$ be sufficiently regular functions defined on $\Omega$ and $\partial \Omega$, respectively.
Our considerations center around the Robin-type boundary value problem
\begin{subequations}\label{eq:robin-gen}
\begin{align}
\Delta u \, &= \, f \quad \text{in } \Omega \, ,\\
\partial_\nu u + au \, &= \, g \quad \text{on } \partial \Omega \, , \label{eq:robin-gen-bd}
\end{align}
\end{subequations}
where $a$ in \eqref{eq:robin-gen-bd} denotes the Robin coefficient.
We impose the standing assumption that $0<a_0 \leq a$.
Let $\omega \subset \Omega$ be a non-empty, open set.

As the inverse Robin problem, we understand the following challenge.
We assume, that the functions $f$ and $g$ in \eqref{eq:robin-gen} as well as the solution to \eqref{eq:robin-gen} restricted to $\omega$, i.e., $q=u|_\omega$ are known.
The aim is to reconstruct the Robin parameter $a$.
A sketch of the geometry is found in Figure~\ref{fig:geom}.
\begin{figure}[t!]
\centering 
\includegraphics[scale=.35]{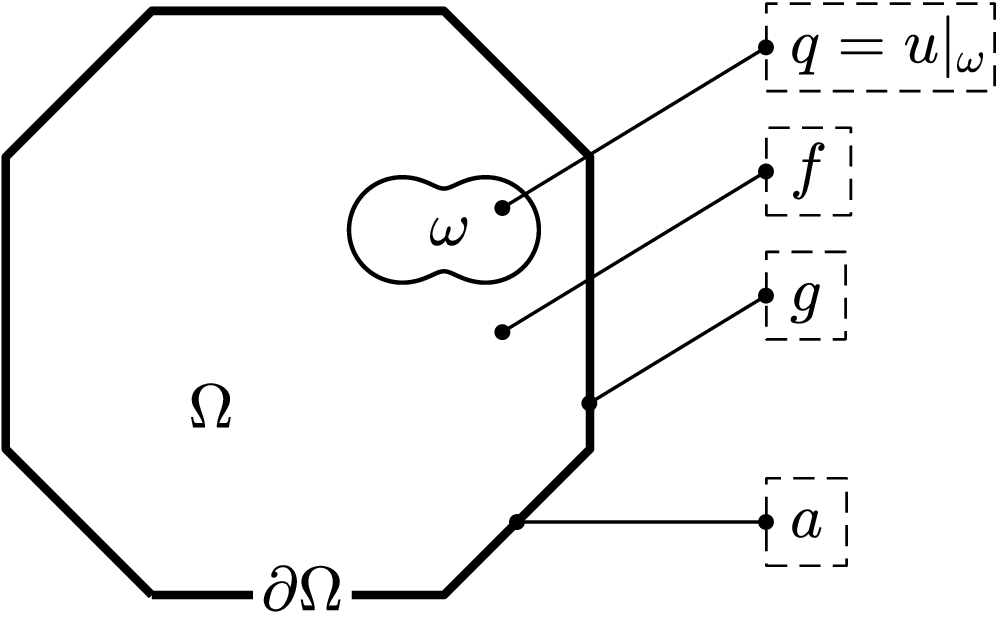}
\caption{Sketch of the geometrical configuration and overview of the functions. For given $f$, $g$ and $q=u|_\omega$ with $u$ solving \eqref{eq:robin-gen}, the aim is to reconstruct $a$.}
\label{fig:geom}
\end{figure}
In other words, for the inverse Robin problem, we study the function $\zeta$ satisfying
\begin{subequations}\label{eq:zetaeq}
\begin{align}
\Delta \zeta \, &= \, f \quad \text{in } \Omega \, ,\\
\partial_\nu \zeta + a\zeta \, &= \, g \quad \text{on } \partial \Omega \, , \\
\zeta|_\omega \, &= \, q \quad \text{in } \omega \, \, .
\end{align}
\end{subequations}
We note that this problem is in general not well-posed for arbitrary functions $f$, $g$ and $q$.
The existence of $\zeta$ must be assumed. 

The inverse Robin problem is in general not solvable.
Indeed, if $\zeta$ solves \eqref{eq:zetaeq} and there is some open, non-empty $\Gamma \subset \partial \Omega$ such that $\zeta = 0$ in $\Gamma$, then $\zeta$ solves \eqref{eq:zetaeq} with $a$ replaced by $a+\kappa$, where $\kappa \in C_0^\infty(\Gamma)$.
This reveals that it is a necessary condition that $\zeta$ does not vanish on any open, non-empty subset of $\Gamma$.
An example for a sufficient condition that can be verified from the data $f$ and $g$ is found in Remark \ref{rem:fandgass} below.
In the literature, other sufficient conditions are given in slightly different contexts, e.g., in \cite{Aletal03, Chaetal04, Sinc15}.

Even if uniqueness of $a$ is guaranteed theoretically, it is not straightforward to recover $a$ in practice.
Unique continuation is in general exponentially unstable. To be precise, it holds that (see \cite[Thm.\@ 5.3]{Aletal09})
\begin{align*}
\Vert u \Vert_{L^2(\Omega)} \, \lesssim \, \Vert u \Vert_{H^1(\Omega)} \log\left( \frac{\Vert u \Vert_{H^1(\Omega)}}{\Vert u \Vert_{L^2(\omega)} + \Vert \Delta u \Vert_{(H^{1}(\Omega))'}}\right)^{-\beta}
\end{align*}
for some $0<\beta < 1$. That no modulus strictly better than single-logarithmic can hold follows from the sharpness of the stability estimate in \cite[Sec.\@ 3]{Bour10} (established in the $(H^1,H^2)$ setting).
This stability bound makes the problem computationally intractable.
Moreover, a numerical approximation of the function $u$ does not carry over the unique continuation properties from the continuum level straightforwardly.
To address these issues we make the following main assumptions in our work. As already outlined, we assume (i) that $u$ vanishes at most at isolated points on $\partial \Omega$.
Secondly, to address the ill-posedness we assume (ii) that the Robin parameter lies in a finite dimensional space denoted by $V_J$, where $J = \mathrm{dim}(V_J)$. 
Later in Section~\ref{sec:pert} we can relax this assumption by requiring instead that $a-\delta \in V_J$, where $\delta$ is a small term carrying, e.g., high frequencies. 
Here, the recovery of $a-\delta$ may be interpreted as a reconstruction of a low pass version of $a$.

\section{Estimates in the continuum}\label{sec:robinintro}
Let $\Omega \subset \R^2$ be an open, bounded curvilinear
and convex domain of class $C^{1,1}$.
Precisely, we let the boundary $\partial \Omega$ of the domain $\Omega$ be defined by
\begin{align}\label{eq:omegadecomp}
\partial \Omega \, = \, \bigcup_{j=1}^N \overline{(\partial \Omega)_j}\, 
\end{align}
for some $N \in \N$, where each $\overline{(\partial \Omega)_j}$ is a $C^{1,1}$ curve with interior $(\partial \Omega)_j$. We assume that the curve $\overline{(\partial \Omega)_{j+1}}$ follows $\overline{(\partial \Omega)_j}$ according to the positive orientation inherited by $\partial \Omega$.
We denote Sobolev spaces of order $s\geq 0$ in the domain $\Omega$ by $H^s(\Omega)$. Since every polygonal domain is a Lipschitz domain, the Sobolev spaces on the boundary $H^s(\partial \Omega)$, $s \ge 0$ can be defined in the standard way by a localization argument (see, e.g., \cite[Ch.\@ 3]{McLean00}).
Their dual spaces are denoted by $H^{-s}(\partial \Omega)$.
In particular, for every function $u \in H^s(\partial \Omega)$ we find that $u|_{(\partial \Omega)_j} \in H^s((\partial \Omega)_j)$.

As motivated in the last section we assume that the Robin coefficient $a$ lies in the finite dimensional subspace $V_J \subset C^1(\partial \Omega)$ that is spanned by the functions $\phi_1, \dots, \phi_J$ for some $J\in \N$.
For $f \in (H^{1}(\Omega))'$ and $g\in H^{-1/2}(\partial \Omega)$ the variational formulation to \eqref{eq:robin-gen}
is to find $u \in H^1(\Omega)$ such that 
\begin{align}\label{eq:weakRobin}
b(u,v) \,=\, \ell_{f,g}(v)\quad \text{for all } v \in H^1(\Omega)\, , 
\end{align}
where $b: H^1(\Omega) \times H^1(\Omega) \to \R$ and $\ell_{f,g}: H^1(\Omega)\to \R$ with
\begin{subequations}\label{eq:bandf}
\begin{align}
b(u,v) \, &= \, \int_\Omega \nabla u \cdot \nabla v \dx + \int_{\partial \Omega} a u v \ds \, ,\\ 
\ell_{f,g}(v) \, &= \, \int_{\partial \Omega} gv \ds-\int_\Omega fv \dx \, .
\end{align}
\end{subequations}
Due to the positivity of $a$, existence and uniqueness of the weak solution $u\in H^1(\Omega)$ follows by the Lax--Milgram theorem. 
This is done in detail, e.g., in \cite[Prop.\@ 31.15]{ErnGue21}, where also the $H^1(\Omega)$ coercivity for $b$
\begin{align}\label{eq:coerc}
\Vert v \Vert_{H^1(\Omega)}^2 \, \lesssim \, b(v,v) \quad \text{for all } v \in H^1(\Omega) 
\end{align}
is proven.
Moreover, the Lax--Milgram theorem provides the bound
\begin{align}\label{eq:wp-bound}
\Vert u \Vert_{H^1(\Omega)} \, \lesssim \, \Vert f \Vert_{(H^{1}(\Omega))'} + \Vert g \Vert_{H^{-1/2}(\partial \Omega)}\, ,
\end{align}
where the implicit constant depends on $\Omega$ and $a_0$ only.
Under the given assumptions on the domain $\Omega$ and additional smoothness requirements on $f$ and $g$ the weak solution to \eqref{eq:robin-gen} is in $H^2(\Omega)$.
This result is proven in \cite[Cor.\@ 3.1]{Mgha92}.
We formulate this result for the convenience of the reader. 
\begin{lemma}\label{lem:H2reg}
Let $f \in L^2(\Omega)$ and let $g \in H^{1/2}(\partial \Omega)$. 
Then, the weak solution to \eqref{eq:robin-gen}
satisfies $u \in H^2(\Omega)$ as well as the bound
\begin{align}\label{eq:uH2bound}
\Vert u \Vert_{H^2(\Omega)} \, \lesssim \, \Vert f \Vert_{L^2(\Omega)} + \Vert g \Vert_{H^{1/2}(\partial \Omega)}\, ,
\end{align}
where the implicit constant depends on $a$.
\end{lemma}
\begin{proof}
We refer to \cite{Mgha92} for the full details on the proof that $u\in H^2(\Omega)$.
In fact, the result is formulated for domains $\Omega$, for which $\partial \Omega$ can be decomposed as in \eqref{eq:omegadecomp} and for functions solving
\begin{subequations}\label{eq:robin-gen3}
\begin{align}
\Delta u \, &= \, f \quad \text{in } \Omega \, ,\\
\partial_{\nu_j} u + au \, &= \, g_j \quad \text{on }  (\partial \Omega)_j \; \text{ for all } j=1,\dots,N \, ,
\end{align}
\end{subequations}
where $\nu_j$ denotes the exterior normal to the segment $(\partial \Omega)_j$ and $g_j \in H^{1/2}((\partial \Omega)_j)$.
To obtain the bound in \eqref{eq:uH2bound}, we observe that the operator 
\begin{align*}
T : H^2(\Omega) \to L^2(\Omega)\times H^{1/2}(\Gamma)\, \quad Tu = (\Delta u, \partial_\nu u + au)
\end{align*}
has a closed graph, since,
by \cite[Thm.\@ 1.5.2.1]{Gris85}, the traces 
\begin{align*}
\partial_{\nu_j} : H^2(\Omega) \to H^{1/2}((\partial \Omega)_j)\quad \text{and} \quad \gamma_j : H^1(\Omega) \to H^{1/2}((\partial \Omega)_j)
\end{align*}
are continuous.
Therefore $T$ is continuous.
Furthermore, due to the well-posedness of \eqref{eq:robin-gen3}, the operator $T$ is bijective.
The bounded inverse theorem now implies that $T^{-1}$ is bounded. This yields the bound in \eqref{eq:uH2bound}.
\end{proof}
To highlight the dependence of a function $u \in H^1(\Omega)$ solving \eqref{eq:robin-gen} on the Robin parameter $a \in V_J$ we write $u^{(a)}$, whenever this is relevant.
For the inverse Robin problem we suppose that we know $f \in L^2(\Omega)$, $g \in H^{1/2}(\partial \Omega)$ as well as $q=u^{(\tilde{a})}|_\omega \in L^2(\Omega)$, where $u^{(\tilde{a})} \in H^1(\Omega)$ solves
\begin{subequations}\label{eq:robin-gen-ex}
\begin{align}
\Delta u^{({a})} \, &= \, f \quad \text{in } \Omega \, ,\\
\partial_\nu u^{({a})} + {a}u^{({a})} \, &= \,  g  \quad \text{on } \partial \Omega \, 
\end{align}
\end{subequations}
with $a$ replaced by $\tilde{a}$.
In our notation we distinguish between the Robin parameters $\tilde{a}$ and $a$. The first one is the true Robin parameter in the inverse Robin problem that we aim to recover, whereas the latter one is some arbitrary $a \in C^1(\partial \Omega)$ with $0<a_0 \leq a$. The corresponding solutions
$u^{(\tilde{a})}$ and $u^{({a})}$ solve \eqref{eq:robin-gen-ex} with parameters $\tilde{a}$ and $a$, respectively.
Due to our motivation from the last section we assume that
$u^{(\tilde{a})}$ vanishes at most at isolated points on $\partial \Omega$.
Our intention is to reconstruct the Robin parameter $\tilde{a}$ from the given data $(f, g, q) \in L^2(\Omega) \times H^{1/2}(\partial \Omega) \times H^1(\omega)$.

\begin{remark}\label{rem:fandgass}
Let $(f,g)\neq (0,0)$. For $f\geq 0$ and $g\leq 0$, one can show that $u^{(\tilde{a})}$ solving \eqref{eq:robin-gen-ex} satisfies $u^{(\tilde{a})}<0$ on $\partial \Omega$.
Similarly, if $f\leq 0$ and $g\geq 0$, then $u^{(\tilde{a})}>0$ on $\partial \Omega$.
For $f\geq 0$ and $g\leq 0$ this can be shown as follows. 
Using $v = u^{(\tilde{a}),+}=\max\{0, u^{(\tilde{a})}\} \geq 0$ in \eqref{eq:weakRobin} we find that $u^{(\tilde{a}),+}$ vanishes in all $\Omega$ and $u = -u^{(\tilde{a}),-}\leq 0$, where $u^{(\tilde{a}),-} = \max\{0, -u^{(\tilde{a})}\}$.
By the strong maximum principle (see, e.g., \cite[Thm.\@ 8.19]{GilTru77}) one finds that $u^{(\tilde{a})}<0$ in $\Omega$. 
If we assume that $\xi\in \partial \Omega$ is such that $u^{(\tilde{a})}(\xi) = 0$, then $\partial_\nu u^{(\tilde{a})}(\xi) > c >0$ since, by the strong maximum principle, if $u^{(\tilde{a})}$ was zero in a neighborhood of $\xi$, it would vanish everywhere.
But now, due to the boundary condition in \eqref{eq:robin-gen-ex}, $g \geq c>0$, which is a contradiction to the assumption.
\end{remark}
Following recent works (see \cite{BurOks24, BurOksZhi25}) we introduce a Lagrangian, which is the function $\Theta : H^1(\Omega) \times H^1(\Omega) \times \{a \in C^1(\partial \Omega) \, : \, 0<a_0 \leq a\} \to \R$ with
\begin{align}\label{eq:Theta}
\Theta(u,z,a) \, = \, \frac{1}{2} \Vert u - q \Vert_{L^2(\omega)}^2 + b(u,z) - \ell_{f,g}(z) \, .
\end{align}
It consists of a sum of a data fidelity term and a constraint arising from the weak formulation \eqref{eq:weakRobin}.
The search for saddle points of this Lagrangian leads to three conditions. The first one is to determine $u \in H^1(\Omega)$ satisfying $b(u,v) = \ell_{f,g}(v)$ for all $v \in H^1(\Omega)$ (derive $\Theta$ w.r.t.\@ $z$).
The second one is to determine $z \in H^1(\Omega)$ satisfying $b(z,v) = \ell_{1_\omega(u-q),0}(v)$ for all $v \in H^1(\Omega)$ (derive $\Theta$ w.r.t.\@ $u$).
Finally, the third one is that $L: C^1(\partial \Omega) \to \R$ defined by
\begin{align}\label{eq:defL}
L(\eta) \, = \, \int_{\partial \Omega} \eta u z \ds 
\end{align}
needs to be zero for all $\eta \in C^1(\partial \Omega)$ (derive $\Theta$ w.r.t.\@ $a$).
This means that $u$ and $z$ must be solutions to Robin-type problems, while $L(\eta) = 0$ for all $\eta \in C^1(\partial \Omega)$ needs to be additionally imposed.
The latter requirement can only be satisfied numerically on a finite dimensional space $V_J \subset C^1(\partial \Omega)$ and with finite element approximations $u_h \approx u$ and $z_h \approx z$.
Under these assumptions we establish a reconstruction scheme based on Newton's method for finding roots for a discretized version of the functional $L$ from \eqref{eq:defL}. 
We prove that retrieving $\tilde{a}$ in this way yields a second order convergent scheme in $h$. We show this for noiseless data in Thm.\@ \ref{thm:main} and for noisy data -- up to the point where the noise level dominates -- in Thm.\@ \ref{thm:mainN}.

The next theorem is a stability result for unique continuation with additional Robin-type constraint. It will be the basis of our convergence analysis.
Denote by $P:H^{-1/2}(\partial \Omega) \to \widetilde{V}_J$ the projection onto $\widetilde{V}_J\subset H^{1/2}(\partial \Omega)$, a finite dimensional subspace. 
Moreover, let $Q = 1-P :  H^{-1/2}(\partial \Omega) \to H^{-1/2}(\partial \Omega)$.
\begin{theorem}\label{thm:stability}
For any $v\in H^1(\Omega)$ with 
$\Delta v \in L^2(\Omega)$ it holds that
    \begin{align}\label{stability}
\Vert v \Vert_{H^1(\Omega)} \, \lesssim \, \Vert v \Vert_{L^2(\omega)} + \Vert Q(\p_\nu v + av)\Vert_{H^{-1/2}(\partial \Omega)} + \Vert\Delta v \Vert_{L^2(\Omega)} \, .
    \end{align}
\end{theorem}
The implicit constant in \eqref{stability} depends on $\Vert a \Vert_{C^1(\partial \Omega)}$ and on $\widetilde{V}_J$.
\begin{proof}
First, we show the stability bound \eqref{stability} for $w\in H^1(\Omega)$ satisfying
\begin{subequations}\label{eq:up1}
\begin{align}
\Delta w \, &= \, 0 \, \text{ in } \Omega \, , \\
 Q(\partial_\nu w + aw) \, &= \, 0 \, \text{ on } \partial \Omega\, 
\end{align}
\end{subequations}
in the weak sense.
Since $w$ is harmonic it holds that $\partial_\nu w \in H^{-1/2}(\partial \Omega)$ (see e.g. \cite[Lem.\@ 4.3]{McLean00}).
We consider the variational formulation of \eqref{eq:up1} and use $w\in H^1(\Omega)$ as a test function, which implies that
\begin{align}\label{eq:urep}
\Vert \nabla w \Vert_{L^2(\Omega)}^2 + \Vert a^{1/2} w \Vert_{L^2(\partial \Omega)}^2 \, = \, \int_{\partial \Omega} P(\partial_\nu w + aw) w \ds \, .
\end{align}
The Cauchy--Schwarz inequality shows that
\begin{align}\label{eq:nablaubound}
\Vert \nabla w \Vert_{L^2(\Omega)}^2 \, \lesssim \, \Vert P(\partial_\nu w +aw) \Vert_{L^2(\partial \Omega)} \Vert w \Vert_{L^2(\partial \Omega)} \, .
\end{align}
Furthermore, the Poincar\'e inequality (see e.g.\@ \cite[Eq.\@ (5.3.3)]{BrenSco08}) together with \eqref{eq:urep}, the Cauchy--Schwarz inequality and the assumption that $0<a_0 \leq a$ yields that
\begin{align}\label{eq:ubound}
\begin{split}
\Vert w \Vert_{L^2(\Omega)}^2 \, &\lesssim \, \frac{a_0}{a_0} \Vert w \Vert_{L^2(\partial \Omega)}^2 + \Vert \nabla w \Vert_{L^2(\Omega)}^2 \\
& \lesssim \, \Vert \nabla w \Vert_{L^2(\Omega)}^2 + \Vert a^{1/2} w \Vert_{L^2(\partial \Omega)}^2 \\
&\lesssim \, \Vert P(\partial_\nu w +aw) \Vert_{L^2(\partial \Omega)} \Vert w \Vert_{L^2(\partial \Omega)} \, .
\end{split}
\end{align}
Adding \eqref{eq:nablaubound} and \eqref{eq:ubound} and using 
the boundedness of the trace from $H^1(\Omega)$ to $H^{1/2}((\partial \Omega)_j)$ (see \cite[Thm.\@ 1.5.2.1]{Gris85}) yields
\begin{align}\label{eq:boundu}
\Vert w \Vert_{H^1(\Omega)} \, \lesssim \, \Vert P(\partial_\nu w +aw) \Vert_{L^2(\partial \Omega)} \, .
\end{align}
Now we use that $P$ is a projection from $H^{-1/2}(\partial \Omega)$ to a finite dimensional subspace of $L^2(\partial \Omega)$, and get for any $0<\varepsilon<1/2$ that
\begin{align}\label{eq:boundP}
\begin{split}
\Vert P(\partial_\nu w +aw) \Vert_{L^2(\partial \Omega)} \, &\lesssim \, \Vert P(\partial_\nu w +aw) \Vert_{H^{-3/2-\varepsilon}(\partial \Omega)} \\
\, &= \, \Vert \partial_\nu w +aw \Vert_{H^{-3/2-\varepsilon}(\partial \Omega)} \\
\, &\lesssim \, \Vert \partial_\nu w \Vert_{H^{-3/2-\varepsilon}(\partial \Omega)} +\Vert aw \Vert_{H^{-1/2-\varepsilon}(\partial\Omega)} \\
\, &\lesssim \, \Vert \partial_\nu w \Vert_{H^{-3/2-\varepsilon}(\partial \Omega)} +\Vert a \Vert_{C^1(\partial \Omega)} \Vert w \Vert_{H^{-1/2-\varepsilon}(\partial\Omega)} \\
\, &\lesssim \, \Vert  w \Vert_{L^2(\Omega)} \, ,
\end{split}
\end{align}
where we used that the trace operators $\partial_{\nu_j} w$ and $\gamma_j w$ can be extended to bounded operators from 
$D(\Delta) = \{w \in L^2(\Omega) \, : \, \Delta w \in L^2(\Omega)\}$ to
$H^{-3/2-\varepsilon}((\partial \Omega)_j)$ and $H^{-1/2-\varepsilon}((\partial \Omega)_j)$, respectively (see\@ \cite[Thm.\@ 1.5.3.4]{Gris85}).
Note that the normal trace $\partial_\nu w$ is in $H^{-1/2}(\partial \Omega)$, however we want to relax the norm further to $H^{-3/2-\varepsilon}(\partial \Omega)$ to end up with $\Vert w \Vert_{L^2(\Omega)}$ in \eqref{eq:boundP}.
From now on we can proceed just as in the proof of \cite[Thm.\@ 3.1]{BurOks24}.
The elliptic unique continuation stability result of \cite[Thm.\@ 5.3]{Aletal09} reads
\begin{align}\label{eq:ucpbound}
\Vert w \Vert_{L^2(\Omega)} \, \leq \, \Vert w \Vert_{H^1(\Omega)} \mu\bigg(\frac{\Vert w \Vert_{L^2(\omega)}}{\Vert w \Vert_{H^1(\Omega)}}\bigg)\, ,
\end{align}
where $\mu(t) \leq C/(\log(1/t)^{\beta})$ for $t<1$ and some $0<\beta<1$.
Now we combine \eqref{eq:boundP} with \eqref{eq:ucpbound} and \eqref{eq:boundu} and see that
\begin{multline*}
\Vert P(\partial_\nu w +aw) \Vert_{L^2(\partial \Omega)} \, \lesssim \, \Vert  w \Vert_{L^2(\Omega)}  \, \lesssim \Vert w \Vert_{H^1(\Omega)} \mu\bigg(\frac{\Vert w \Vert_{L^2(\omega)}}{\Vert w \Vert_{H^1(\Omega)}}\bigg) \\
\, \lesssim \Vert P(\partial_\nu w +aw) \Vert_{L^2(\partial \Omega)} \mu\bigg(\frac{\Vert w \Vert_{L^2(\omega)}}{\Vert w \Vert_{H^1(\Omega)}}\bigg)\, .
\end{multline*}
This shows that 
$1 \, \lesssim \, \mu({\Vert w \Vert_{L^2(\omega)}}/{\Vert w \Vert_{H^1(\Omega)}})$. 
From the bound of $\mu$ one can now see that
\begin{align}\label{eq:prebound}
\Vert w \Vert_{H^1(\Omega)} \, \lesssim \Vert w \Vert_{L^2(\omega)}\, .
\end{align}
This shows the bound \eqref{stability}.

Now, let $v\in H^1(\Omega)$ with 
$\Delta v \in L^2(\Omega)$ be arbitrary.
We consider the unique (weak) solution $y \in H^1(\Omega)$ of 
\begin{align*}
\Delta y \, = \, \Delta v \, \text{ in } \Omega \, , \quad \partial_\nu y + ay \, = \, Q(\partial_\nu v + av) \, \text{ on } \partial \Omega\, .
\end{align*}
The function $w=y-v \in H^1(\Omega)$ satisfies \eqref{eq:up1}
and therefore, the bound in \eqref{eq:prebound} applies for $w$.
Now, we use the triangle inequality,  \eqref{eq:prebound} as well as the well-posedness estimate in \eqref{eq:wp-bound} and see that
\begin{align*}
\Vert v \Vert_{H^1(\Omega)} \, &\lesssim \, \Vert y \Vert_{H^1(\Omega)} + \Vert w \Vert_{L^2(\omega)} \\
&\lesssim \, \Vert y \Vert_{H^1(\Omega)} + \Vert v \Vert_{L^2(\omega)} \\
\, &\lesssim \, \Vert \Delta v \Vert_{L^2(\Omega)} + \Vert Q(\partial_\nu v+av) \Vert_{H^{-1/2}(\partial \Omega)} + \Vert v \Vert_{L^2(\omega)} \, .
\end{align*}
\end{proof}
\begin{remark}
Our considerations are limited to the two-dimensional setting since both the higher regularity from Lemma~\ref{lem:H2reg} and the stability Theorem~\ref{thm:stability} use results that are only available in two spatial dimensions.
\end{remark}
\begin{remark}
The stability Theorem~\ref{thm:stability} can also be shown with the $(H^{1}(\Omega))'$ norm replacing the $L^2(\Omega)$ norm in the right hand side of \eqref{stability} under the assumption that $\Delta v \in (H^{1}(\Omega))'$ and $\partial_\nu v \in H^{-1/2}(\partial \Omega)$ is meaningful. One might, e.g., define it as
\begin{align*}
\int_{\Omega} \Delta v w \dx \, = \, -\int_{\Omega} \nabla v \cdot \nabla w \dx + \int_{\partial \Omega}\partial_\nu v w \ds \, \quad \text{for } w \in H^1(\Omega)\,.
\end{align*}

\end{remark}

\section{Linearization with respect to the Robin parameter}
In this section we study linearizations of Robin boundary value problems with respect to the Robin parameter. 
This is necessary in order to compute Fr\'echet derivatives that are needed for our derivative-based reconstruction algorithm.

From now on let $\Omega$ be a convex, open and bounded polygonal domain.
We denote by $\mathcal{T}_h$ a triangulation of $\Omega$ consisting of simplices that form a simplicial complex. 
Moreover, we denote by $K\subset \R^2$ a simplex and define $h = \max_{K \in \mathcal{T}_h}\text{diam}(K)$.
Let $h_0>0$ be sufficiently small so that for all triangulations with mesh size $h\leq h_0$ the standard finite element approximations hold.
We introduce the finite element space
\begin{align*}
V_h \, = \, \{ u \in H^1(\Omega) \, : \, u|_K \in \mathbb{P}^1 \; \text{for all } K \in \mathcal{T}_h \}\, ,
\end{align*}
where $\mathbb{P}^1$ denotes the space of first order polynomials on $K$,
and consider the bilinear forms $b$ and $\ell_{f,g}$ from \eqref{eq:bandf}.
The finite element method that approximates a solution to \eqref{eq:robin-gen-ex} consists in determining $u_h^{(a)} \in V_h$ such that
\begin{align}\label{eq:feform}
b(u_h^{(a)}, v) \, = \, \ell_{f,g}(v) \quad \text{for all } v \in V_h\,.
\end{align}
Due to the coercivity bound in \eqref{eq:coerc}, $u_h^{(a)} \in V_h$ exists, is uniquely determined and satisfies the same a-priori bound as $u$ in \eqref{eq:wp-bound} (see also \cite[Lem. 26.3]{ErnGue21}).
We provide two error bounds for this approximation in the next lemma. It follows from the general $H^1$-estimate that can be found, e.g., in \cite[Thm.\@ 3.16]{ErnGue04} and an application of Nitsche's trick.
\begin{lemma}\label{lem:fembound}
Let $u^{(a)} \in H^1(\Omega)$ be the unique solution to \eqref{eq:robin-gen-ex} and let $u_h^{(a)} \in V_h$ be determined by \eqref{eq:feform}. Then, the approximation error can be bounded via
\begin{align*}
\Vert u^{(a)} - u_h^{(a)} \Vert_{L^2(\Omega)} + h\Vert u^{(a)} - u_h^{(a)} \Vert_{H^1(\Omega)} \, \lesssim \, h^2 \Vert u^{(a)} \Vert_{H^2(\Omega)}\, .
\end{align*}
\end{lemma}

We start to derive linearizations of $u_h^{(a)}$ with respect to the Robin parameter $a$.
For $a \in C^1(\partial \Omega)$ with $0<a_0\leq a$ let $a \mapsto u_h^{(a)}$ be the operator that maps the Robin parameter $a$ to the solution $u_h^{(a)} \in V_h$ of \eqref{eq:feform}. 
Note that $b$ from \eqref{eq:bandf} also depends on $a$.
For a fixed $a\in C^1(\partial \Omega)$ let 
\begin{align}\label{def:Lambda}
U_a : D(U_a) \subset C^1(\partial \Omega) \to V_h\, , \quad U_a(\eta) \, = \,  u_h^{(a+\eta)}\, ,
\end{align}
where $D(U_a)$ is a neighborhood of the zero function in $C^1(\partial \Omega)$ that is so small that $u_h^{(a+\eta)}$ is well-defined for any $\eta\in D(U_a)$.
The operator $U_a$ maps a perturbation of the Robin parameter $\eta$ to the solution of the weak formulation with perturbed Robin parameter $a+\eta$.
Note that $U_a(0) = u_h^{(a)}$.
We are interested in characterizing the Fr\'echet derivative of the operator $U_a$ at zero, i.e., the operator $U_a'(0): C^1(\partial \Omega) \to H^1(\Omega)$ that satisfies
\begin{align*}
\frac{1}{\Vert \eta \Vert_{C^1(\partial \Omega)}} \Vert U_a(\eta) - U_a(0) - U_a'(0)\eta \Vert_{H^1(\Omega)} \, \to \, 0 \quad \text{as } \Vert \eta \Vert_{C^1(\partial \Omega)} \to 0\, . 
\end{align*}
We start with a result on the continuity of $U_a$.
\begin{lemma}\label{lem:continuity}
For $a \in C^1(\partial \Omega)$ with $0<a_0\leq a$ the operator $U_a$ from \eqref{def:Lambda} satisfies
\begin{align*}
\Vert U_a(\eta_1) - U_a(\eta_2)\Vert_{H^1(\Omega)}\, \lesssim \, \Vert \eta_1 - \eta_2 \Vert_{C^1(\partial \Omega)}
\end{align*}
for any $\eta_1 \in D(U_a)$ and  $\eta_2 \in D(U_a)$ with $\Vert \eta_2 \Vert_{C^1(\partial \Omega)} \leq \rho$ for some sufficiently small $\rho>0$.
\end{lemma}
\begin{proof}
We consider the weak formulations for both $U_a(\eta_1) = u_h^{(a+\eta_1)}$ and $U_a(\eta_2) = u_h^{(a+\eta_2)}$, subtract them and obtain
\begin{multline}\label{eq:differencewform}
\int_\Omega \nabla( u_h^{(a+\eta_1)} - u_h^{(a+\eta_2)} ) \cdot \nabla v \dx \\
+ \int_{\partial \Omega} ( (a+\eta_1) u_h^{(a+\eta_1)} - (a+\eta_2) u_h^{(a+\eta_2)} ) v \ds \, = \, 0 \quad \text{for all } v \in V_h\, .
\end{multline}
As a test function we insert $v = u_h^{(a+\eta_1)} - u_h^{(a+\eta_2)}$ and obtain
\begin{align}\label{eq:instest}
\Vert \nabla v \Vert_{L^2(\Omega)}^2 
+ \Vert a^{1/2} v \Vert_{L^2(\partial \Omega)}^2 
\, = \, \int_{\partial \Omega}(\eta_1 u_h^{(a+\eta_1)} - \eta_2 u_h^{(a+\eta_2)})  v \ds  \, .
\end{align}
The left hand side in \eqref{eq:instest} is simply $b(v,v)$.
Moreover, the right hand side in \eqref{eq:instest} can be bounded by
\begin{multline*}
\left|\int_{\partial \Omega}(\eta_1 u_h^{(a+\eta_1)} - \eta_2 u_h^{(a+\eta_2)})  v \ds\right| 
\, = \, \left| \int_{\partial \Omega} ((\eta_1-\eta_2)u^{(a+\eta_1)} + \eta_2 v) v \ds \right| \\
\leq \, \Vert \eta_1 - \eta_2 \Vert_{C^1(\partial \Omega)}\Vert u_h^{(a+\eta_1)} \Vert_{H^1(\Omega)}\Vert v \Vert_{H^1(\Omega)} + \Vert \eta_2 \Vert_{C^1(\partial \Omega)}\Vert v \Vert_{H^1(\Omega)}^2\, .
\end{multline*}
We apply the coercivity result from \eqref{eq:coerc} together with the previous bound and see that \eqref{eq:instest} implies that
\begin{align}\label{eq:vbound}
\Vert v \Vert_{H^1(\Omega)} \, \lesssim \, \Vert \eta_1 - \eta_2 \Vert_{C^1(\partial \Omega)}\Vert u_h^{(a+\eta_1)} \Vert_{H^1(\Omega)} + \Vert \eta_2 \Vert_{C^1(\partial \Omega)}\Vert v \Vert_{H^1(\Omega)}\, .
\end{align} 
We use the assumption that $\Vert \eta_2 \Vert_{C^1(\partial \Omega)}$ is sufficiently small and rearrange \eqref{eq:vbound} to get that
\begin{align}\label{eq:vbound2}
\Vert u_h^{(a+\eta_1)} - u_h^{(a+\eta_2)} \Vert_{H^1(\Omega)} \, \lesssim \, \Vert \eta_1 - \eta_2 \Vert_{C^1(\partial \Omega)}\Vert u_h^{(a+\eta_1)} \Vert_{H^1(\Omega)} \, .
\end{align}
Finally, the weak formulation for $u_h^{(a+\eta_1)}$ gives us that
\begin{align}\label{eq:uhetabound}
\Vert u_h^{(a+\eta_1)} \Vert_{L^2(\partial \Omega)} \, \lesssim \, \Vert f \Vert_{L^2(\Omega)} + \Vert g \Vert_{H^{1/2}(\partial \Omega)}\, .
\end{align}
Combining \eqref{eq:vbound2} and \eqref{eq:uhetabound} now yields the result.
\end{proof}
The Fr\'echet derivative $U_a'(0)$ in direction $\eta$ is characterized through the finite element solution corresponding to the boundary value problem
\begin{subequations}
\begin{align}
\Delta \dot{u}^{(a)}\, &= \, 0 \quad \text{in } \Omega \, ,\\
\partial_\nu \dot{u}^{(a)} + a\dot{u}^{(a)} \, &= \, -\eta u_h^{(a)} \quad \text{on } \partial \Omega \, , \label{eq:bcu}
\end{align}
\end{subequations}
as the next theorem shows.

\begin{theorem}\label{thm:fderu}
Let $\dot{u}_h^{(a)} \in V_h$ denote the unique solution of
\begin{align}\label{eq:uhdot}
b(\dot{u}_h^{(a)}, v) \,=\, \ell_{0,-\eta u_h^{(a)}}(v) \quad \text{for all } v \in V_h\, , 
\end{align}
where $u_h^{(a)}$ is defined by
$b({u}_h^{(a)}, v) = \ell_{f,g}(v)$ for all $v \in V_h$. 
Then,
$\dot{u}_h^{(a)}= U_a'(0)\eta$.
\end{theorem}
\begin{proof}
Similarly to the proof of Lemma~\ref{lem:continuity}, we write down the weak formulation for $u_h^{(a+\eta)}$ and subtract the weak formulation for $u_h^{(a)}$ as well as the one from $\dot{u}_h^{(a)}$ from \eqref{eq:uhdot}. This results in
\begin{multline}\label{eq:differencewform2}
\int_\Omega \nabla( u_h^{(a+\eta)} - u_h^{(a)} - \dot{u}_h^{(a)}) \cdot \nabla v \dx + \int_{\partial \Omega} a (u_h^{(a+\eta)} -  u_h^{(a)} - \dot{u}_h^{(a)} ) v \ds \\
+\int_{\partial \Omega} \eta (u_h^{(a+\eta)} - u_h^{(a)}) v \ds  \, = \, 0 \quad \text{for all } v \in V_h\, .
\end{multline}
We insert the function $v = u_h^{(a+\eta)} - u_h^{(a)} - \dot{u}_h$ for $v$ in \eqref{eq:differencewform2}, apply the Cauchy--Schwarz inequality and subsequently Lemma~\ref{lem:continuity}, which yields that
\begin{align*}
b(v,v) 
\, &\lesssim \, \left| \int_{\partial \Omega} \eta (u_h^{(a+\eta)} - u_h^{(a)}) v \ds \right| \\
\, &\lesssim \, \Vert \eta \Vert_{C^1(\partial \Omega)} \Vert u_h^{(a+\eta)} - u_h^{(a)} \Vert_{L^2(\partial \Omega)} \Vert v\Vert_{L^2(\partial \Omega)}  
\ \lesssim \, \Vert \eta \Vert_{C^1(\partial \Omega)}^2 \Vert v \Vert_{H^1(\Omega)}\, .
\end{align*}
Applying the coercivity \eqref{eq:coerc} and dividing by $\Vert v \Vert_{H^1(\Omega)}$ yields the result.
\end{proof}
Theorem~\ref{thm:stability} yields the following corollary. Recall the definition of the finite dimensional space $V_J$ from the beginning of Section~\ref{sec:robinintro}.
\begin{corollary}\label{cor:dotuh}
Let $\eta \in V_J$ be the function appearing in \eqref{eq:uhdot}. 
Then, the function $\dot{u}_h^{(a)} \in V_h$ defined by \eqref{eq:uhdot} satisfies
\begin{align}\label{eq:uhpbound}
\Vert \dot{u}_h^{(a)} \Vert_{H^1(\Omega)}\, \lesssim \, \Vert \dot{u}_h^{(a)} \Vert_{L^2(\omega)} + h \Vert \eta \Vert_{C^1(\partial \Omega)} \, .
\end{align}
The implicit constant in \eqref{eq:uhpbound} depends on $\Vert a \Vert_{C^1(\partial \Omega)}$ and on $V_J$.
\end{corollary}
\begin{proof}
To apply Theorem~\ref{thm:stability}, let $P: H^{-1/2}(\partial \Omega) \to \widetilde{V}_J^{(u^{(a)})}$, where $\widetilde{V}_J^{(u^{(a)})}$ is the space spanned by $\phi_1u^{(a)}, \dots, \phi_Ju^{(a)}$ and $Q=1-P$.
We introduce the auxiliary function $u^{\prime (a)} \in H^1(\Omega)$ satisfying
\begin{subequations}\label{eq:robinfrech3}
\begin{align}
\Delta u^{\prime (a)} \, &= \, 0 \quad \text{in } \Omega \, ,\\
\partial_\nu u^{\prime (a)}  + au^{\prime (a)}  \, &= \, -\eta u^{(a)} \quad \text{on } \partial \Omega \, . 
\end{align}
\end{subequations}
Since $\eta u^{(a)} \in \widetilde{V}_J^{(u^{(a)})}$ Theorem~\ref{thm:stability} applied to $u^{\prime (a)} $ from \eqref{eq:robinfrech3} gives that
\begin{align}\label{eq:stabapp}
\Vert u^{\prime (a)}  \Vert_{H^1(\Omega)} \, \lesssim \, \Vert u^{\prime (a)}  \Vert_{L^2(\omega)}\, .
\end{align}
Applying the triangle inequality twice and using \eqref{eq:stabapp} gives
\begin{align}\label{eq:estuhdot}
\begin{split}
\Vert \dot{u}_h^{(a)} \Vert_{H^1(\Omega)} \, &\leq \, \Vert \dot{u}_h^{(a)} - u^{\prime (a)} \Vert_{H^1(\Omega)} + \Vert  u^{\prime (a)} \Vert_{H^1(\Omega)} \\
\, &\lesssim \, \Vert \dot{u}_h^{(a)} - u^{\prime (a)} \Vert_{H^1(\Omega)} + \Vert  \dot{u}_h^{(a)} \Vert_{L^2(\omega)} \, .
\end{split}
\end{align}
Another use of the triangle inequality shows that
\begin{align}\label{eq:uprimeest}
\Vert \dot{u}_h^{(a)} - u^{\prime (a)} \Vert_{H^1(\Omega)} \, \leq \Vert \dot{u}^{(a)} - \dot{u}_h^{(a)}\Vert_{H^1(\Omega)} + \Vert \dot{u}^{(a)}  - u^{\prime (a)} \Vert_{H^1(\Omega)} \, .
\end{align}
The $H^2$-regularity estimate from Lemma~\ref{lem:H2reg} together with a twofold application of Lemma~\ref{lem:fembound} yields 
\begin{align*}
\Vert \dot{u}^{(a)} - \dot{u}_h^{(a)}\Vert_{H^1(\Omega)}\, \lesssim \, h\Vert \eta \Vert_{C^1(\partial \Omega)} \Vert u_h^{(a)} \Vert_{H^{1/2}(\partial \Omega)} \, \lesssim \, h \Vert \eta \Vert_{C^1(\partial \Omega)} \Vert u^{(a)} \Vert_{H^2(\Omega)} \, .
\end{align*}
In the same way, Lemma~\ref{lem:H2reg} together with Lemma~\ref{lem:fembound} show that
\begin{align*}
\Vert \dot{u}^{(a)}  - u^{\prime (a)} \Vert_{H^1(\Omega)} \, \lesssim \, h \Vert \eta \Vert_{C^1(\partial \Omega)} \Vert u^{(a)} \Vert_{H^2(\Omega)}\, .
\end{align*}
Using these two estimates in \eqref{eq:uprimeest} and combining this with \eqref{eq:estuhdot} finally yields \eqref{eq:uhpbound}.
\end{proof}
We return to the setting of our inverse Robin problem introduced in Section~\ref{sec:robinintro}, in which we know 
$q= u^{(\tilde{a})}|_\omega$, together with $f\in L^2(\Omega)$ and $g\in H^{1/2}(\partial \Omega)$. 
Based on the motivation from Section~\ref{sec:robinintro} we introduce a second Robin type boundary value problem, which is to find $z^{(a)} \in H^1(\Omega)$ such that
\begin{subequations}\label{eq:za}
\begin{align}
\Delta z^{(a)} \, &= \, 1_\omega(u^{(a)}-q) \quad \text{in } \Omega \, ,\\
\partial_\nu z^{(a)} + az^{(a)} \, &= \, 0 \quad \text{on } \partial \Omega \, .
\end{align}
\end{subequations}
Clearly, for $a = \widetilde{a}$, we find that $z^{(\widetilde{a})}=0$.
Furthermore, let $z_h^{(a)} \in V_h$ be the unique solution to 
\begin{align}\label{eq:zh}
b(z_h^{(a)}, v ) \, = \, \ell_{1_\omega(u_h^{(a)}-q),0}(v) \quad \text{for all } v \in V_h\, . 
\end{align}
Note that in the right hand side we impose $u_h^{(a)}$ instead of $u^{(a)}$.
For $z_h^{(\tilde{a})} \in V_h$ we obtain the following bound.
\begin{lemma}\label{lem:zhbound}
Let $z_h^{(\tilde{a})} \in V_h$ be the unique solution to \eqref{eq:zh} for $a=\tilde{a}$. Then,
\begin{align}\label{eq:zhtildebound}
\Vert z_h^{(\tilde{a})} \Vert_{H^1(\Omega)} \, \lesssim \, h^2\, .
\end{align}
The implicit constant in \eqref{eq:zhtildebound}depends on $\tilde{a}$.
\end{lemma}
\begin{proof}
We introduce the function $\check{z}^{(\tilde{a})} \in H^1(\Omega)$ solving
\begin{subequations}\label{eq:checkz}
\begin{align}
\Delta \check{z}^{(\tilde{a})} \, &= \, 1_\omega(u_h^{(\tilde{a})}-q) \quad \text{in } \Omega \, ,\\
\partial_\nu \check{z}^{(\tilde{a})} + \tilde{a}\check{z}^{(\tilde{a})} \, &= \, 0 \quad \text{on } \partial \Omega \, 
\end{align}
\end{subequations}
and find by using the well-posedness estimate \eqref{eq:wp-bound} and Lemma~\ref{lem:fembound} that
\begin{align}\label{eq:zcheckest}
\Vert \check{z}^{(\tilde{a})} \Vert_{H^1(\Omega)} \, \lesssim \, \Vert u_h^{(\tilde{a})}-q \Vert_{L^2(\omega)} \, \lesssim \, h^2 \Vert u^{(\tilde{a})} \Vert_{H^2(\Omega)}\, .
\end{align}
Moreover, another application of Lemma~\ref{lem:fembound} and the $H^2$-bound \eqref{eq:uH2bound} give
\begin{align}\label{eq:zhapprox}
\Vert \check{z}^{(\tilde{a})} - z_h^{(\tilde{a})} \Vert_{H^1(\Omega)} \, \lesssim \, h \Vert \check{z}^{(\tilde{a})} \Vert_{H^2(\Omega)} \, \lesssim \, h \Vert u_h^{(\tilde{a})}-q \Vert_{L^2(\omega)} \, \lesssim \, h^3 \Vert u^{(\tilde{a})} \Vert_{H^2(\Omega)} \, .
\end{align}
Combining \eqref{eq:zcheckest} and \eqref{eq:zhapprox} now yields
\begin{align*}
\Vert z_h^{(\tilde{a})} \Vert_{H^1(\Omega)} \, \lesssim \, \Vert \check{z}^{(\tilde{a})} \Vert_{H^1(\Omega)} + \Vert \check{z}^{(\tilde{a})} - z_h^{(\tilde{a})} \Vert_{H^1(\Omega)}  
\, \lesssim \, h^2 \Vert u^{(\tilde{a})} \Vert_{H^2(\Omega)}\, .
\end{align*}
\end{proof}

We use the same framework as before and consider the operator $ a \mapsto z_h^{(a)}$ and 
for a fixed $a$, the operator 
\begin{align}\label{def:Lambdaz}
Z_a : D(Z_a) \subset C^1(\partial \Omega) \to V_h\, , \quad Z_a(\eta) \, = \,  z_h^{(a+\eta)}\, ,
\end{align}
where $D(Z_a)$ is a neighborhood of the zero function in $C^1(\partial \Omega)$ that is so small that $z_h^{(a+\eta)}$ is well-defined for any $\eta\in D(Z_a)$.
We can show the following theorem on the continuity and Fr\'echet derivative of $Z_a(0)$.
\begin{theorem}\label{thm:frechz}
For $a \in C^1(\partial \Omega)$ with $0<a_0\leq a$ the operator $Z_a$ from \eqref{def:Lambdaz} satisfies
\begin{align}\label{eq:contdotz}
\Vert Z_a(\eta_1) - Z_a(\eta_2)\Vert_{H^1(\Omega)}\, \lesssim \, \Vert \eta_1 - \eta_2 \Vert_{C^1(\partial \Omega)}
\end{align}
for any $\eta_1 \in D(Z_a)$ and  $\eta_2 \in D(Z_a)$ with $\Vert \eta_2 \Vert_{C^1(\partial \Omega)} \leq \rho$ for some sufficiently small $\rho>0$.
Furthermore, $Z_a'(0)\eta = \dot{z}_h^{(a)}$, where $\dot{z}_h^{(a)} \in V_h$ is the unique solution to 
\begin{align}\label{eq:zhdot}
b(\dot{z}_h^{(a)}, v) \, = \, \ell_{1_\omega\dot{u}_h^{(a)}, - \eta z_h^{(a)}}(v) \, \quad \text{for all } v \in V_h\, , 
\end{align}
where $\dot{u}_h^{(a)} \in V_h$ is the function defined in Theorem~\ref{thm:fderu}.
\end{theorem}

\begin{proof}
The proof can be done similarly to the proofs of Lemma~\ref{lem:continuity} and Theorem~\ref{thm:fderu}:
To prove \eqref{eq:contdotz} one subtracts the weak formulations for $z_h^{(a+\eta_1)}$ and $z_h^{(a+\eta_2)}$ and uses $v = z_h^{(a+\eta_1)} - z_h^{(a+\eta_2)}$ as a test function.
One obtains an expression as the one in \eqref{eq:differencewform} with an additional right hand side, which is 
\begin{align*}
\int_{\omega} (u_h^{(a+\eta_1)} - u_h^{(a+\eta_2)})(z_h^{(a+\eta_1)} - z_h^{(a+\eta_2)}) \dx \, .
\end{align*}
Using Lemma~\ref{lem:continuity} this term can be bounded by
\begin{align*}
\left| \int_{\omega} (u_h^{(a+\eta_1)} - u_h^{(a+\eta_2)})v \dx\right| \,\lesssim \, \Vert \eta_1-\eta_2 \Vert_{C^1(\partial \Omega)} \Vert v \Vert_{H^1(\Omega)}
\end{align*}
and one can proceed as in the proof Lemma~\ref{lem:continuity}, which proves \eqref{eq:contdotz}.
A similar additional consideration easily proves that $Z_a'(0)\eta = \dot{z}_h$.
\end{proof}
\begin{corollary}\label{cor:dotzh}
The function $\dot{z}^{(\tilde{a})}_h \in V_h$ defined in \eqref{eq:zhdot} satisfies
\begin{align}\label{eq:dotzhbound}
\Vert \dot{z}^{(\tilde{a})}_h \Vert_{H^1(\Omega)} \, \lesssim \, \Vert \dot{u}^{(\tilde{a})}_h \Vert_{L^2(\omega)} + h^2 \Vert \eta \Vert_{C^1(\partial \Omega)} \, .
\end{align}
The implicit constant in \eqref{eq:dotzhbound} depends on $\tilde{a}$.
\end{corollary}
\begin{proof}
We recall the function $\check{z}^{(\tilde{a})} \in H^1(\Omega)$ from \eqref{eq:checkz}.
Now we use the well-posedness estimate for $\dot{z}^{(\tilde{a})}_h$ and get that
\begin{align*}
\Vert \dot{z}^{(\tilde{a})}_h \Vert_{H^1(\Omega)} \, &\lesssim \, \Vert \dot{u}^{(\tilde{a})}_h \Vert_{L^2(\omega)} + \Vert \eta \Vert_{C^1(\partial \Omega)} \Vert z_h^{(\tilde{a})} \Vert_{H^{1/2}(\partial \Omega)} \\
& \lesssim \, \Vert \dot{u}^{(\tilde{a})}_h \Vert_{L^2(\omega)} +  \Vert \eta \Vert_{C^1(\partial \Omega)} ( \Vert z_h^{(\tilde{a})} - \check{z}^{(\tilde{a})} \Vert_{H^1(\Omega)} + \Vert \check{z}^{(\tilde{a})} \Vert_{H^1(\Omega)} )\, .
\end{align*}
Finally, \eqref{eq:zcheckest} and \eqref{eq:zhapprox} yield the desired estimate.
\end{proof}
\section{Stable recovery of the Robin parameter}\label{sec:robinrec}
For any $a \in V_J$ with $0 < a_0 \leq a$ we define the functional 
\begin{align}\label{eq:F}
F(a) \, = \,  (F_1(a), \dots, F_J(a) )\, , \; \; \text{where } F_j(a) \, = \, \int_{\partial \Omega} \phi_j u^{(a)} z^{(a)} \ds\, ,
\end{align}
where $u^{(a)} \in H^1(\Omega)$ and $z^{(a)} \in H^1(\Omega)$ solve \eqref{eq:robin-gen-ex} and \eqref{eq:za}, respectively. 
This functional arises from the study of saddle points of the Lagrangian $\Theta$ defined in \eqref{eq:Theta}, precisely, from the condition that the functional $L$ from \eqref{eq:defL} needs to vanish for all functions in $V_J$.
Moreover, we define its discretized counterpart
\begin{align}\label{eq:Fh}
F_h(a) \, = \,  (F_{h,1}(a), \dots, F_{h,J}(a) )\, , \; \; \text{where } F_{h,j}(a) \, = \, \int_{\partial \Omega} \phi_j u_h^{(a)} z_h^{(a)} \ds\, ,
\end{align}
where $u_h^{(a)} \in V_h$ and $z_h^{(a)} \in V_h$ solve \eqref{eq:feform} 
and \eqref{eq:zh}, respectively.
Note that $F_h$ can be computed given $f, q \in L^2(\Omega)$ and $g \in H^{1/2}(\partial \Omega)$. 
Using the framework from the previous section, we can define the Fr\'echet derivative $\dot{F}_h[a]$ of $F_h$ at $a$, which is given by
\begin{subequations}\label{eq:dotFh}
\begin{align}
&\dot{F}_h[a]\eta \, = \,  (\dot{F}_{h,1}[a]\eta, \dots, \dot{F}_{h,J}[a]\eta )\, ,\quad  \text{where} \\
&\dot{F}_{h,j}[a]\eta \, = \, \int_{\partial \Omega} \phi_j (\dot{u}_h^{(a)} z_h^{(a)} + u_h^{(a)} \dot{z}_h^{(a)}) \ds\, 
\end{align}
\end{subequations}
and where $\dot{u}_h^{(a)} \in V_h$ and $\dot{z}_h^{(a)} \in V_h$ are defined in \eqref{eq:uhdot} and \eqref{eq:zhdot}, respectively.

We follow the framework of \cite{keller75} and show that under the assumption that $\tilde{a} \in V_J$ and 
that $u^{(\tilde{a})}$ vanishes at most at isolated points on $\partial \Omega$
and with $B_\rho(a) = \{\alpha \in V_J \, : \, \Vert a-\alpha \Vert_{C^1(\partial \Omega)}\leq \rho\}$ there holds
\begin{itemize}
\item[(i)] $\dot{F}_h$ is uniformly Lipschitz continuous on $B_\rho(\tilde{a})$ for some $\rho>0$,
\item[(ii)] $F_h$ is consistent with $F$ of order 2 at $\tilde{a}$, that is, 
\begin{align}\label{eq:F_hcons2}
\left| F_h(\tilde{a}) \right|_{\R^J} \, \lesssim \, h^2\, .
\end{align}
\item[(iii)] the Fr\'echet derivatives $\dot{F}_h[\tilde{a}]$ have uniformly bounded inverses for mesh sizes $0<h\leq h_0$.
\end{itemize}
Then, by \cite[Thm.\@ 3.6]{keller75} and \cite[Thm.\@ 3.7]{keller75} the following theorem on the stable reconstruction of $\tilde{a}$ follows.
\begin{theorem}\label{thm:main}
Let $\tilde{a} \in V_J$ and let $u^{(\tilde{a})} \in H^1(\Omega)$ denote the solution to \eqref{eq:robin-gen-ex} for $a=\tilde{a}$ and for some $f\in L^2(\Omega)$ and $g \in H^{1/2}(\partial \Omega)$. 
Assume that
$u^{(\tilde{a})}$ vanishes at most at isolated points on $\partial \Omega$.
Moreover, let $q=u^{(\tilde{a})}|_\omega$ for some $\omega \subset \Omega$ and $f$ and $g$ be known. 
Then, 
\begin{itemize}
\item[(a)] For some sufficiently small $0<h \leq h_0$ and $\rho>0$, there is a unique $a_h \in B_\rho(\tilde{a})$ such that $F_h(a_h) = 0$. Furthermore,
\begin{align}\label{eq:h2bound}
\Vert \tilde{a} - a_h \Vert_{C^1(\partial \Omega)} \, \lesssim \, h^2\, .
\end{align}
The implicit constant in \eqref{eq:h2bound} depends on $\Vert \tilde{a} \Vert_{C^1(\partial \Omega)}$ and on $V_J$.
\item[(b)] For an initial guess that is sufficiently close to the unique root $a_h \in B_\rho(\tilde{a})$ from part (a) the Newton iterates corresponding to the equation $F_h(x) = 0$ converge quadratically to $a_h $.
\end{itemize}
\end{theorem}
The next result is a consequence of Lemma~\ref{lem:continuity} and Theorem~\ref{thm:frechz} and treats the uniform Lipschitz continuity of $\dot{F}_h$.
\begin{corollary}\label{cor:contF}
For $a \in V_J$ with $0<a_0\leq a$ the function $\dot{F}_h[\cdot]$ is uniformly Lipschitz continuous on 
$B_\rho(a)$ for some $\rho>0$, that is,
\begin{align*}
\Vert \dot{F}_h[\alpha_1] - \dot{F}_h[\alpha_2] \Vert_{\R^J \leftarrow V_J} \, \lesssim \, \Vert \alpha_1 - \alpha_2 \Vert_{C^1(\partial \Omega)} 
\end{align*} 
for any $\alpha_1, \alpha_2 \in B_{\rho}(a)$.
\end{corollary}
\begin{proof}
We can write $\alpha_1 = a + \eta_1$ and $\alpha_2 = a+\eta_2$ for some $\eta_1, \eta_2 \in B_\rho(0)$.
Using the same procedure as in Lemma~\ref{lem:continuity} and Theorem~\ref{thm:frechz} we can see by using Lemma~\ref{lem:continuity} that
\begin{align}\label{eq:contdotu}
\Vert \dot{u}_h^{(a+\eta_1)} - \dot{u}_h^{(a+\eta_2)} \Vert_{H^1(\Omega)} \, \lesssim \, \Vert \alpha_1 - \alpha_2 \Vert_{C^1(\partial \Omega)} \,.
\end{align}
Similarly, by using Theorem~\ref{thm:frechz} and \eqref{eq:contdotu} one finds that
\begin{align}\label{eq:contdotz2}
\Vert \dot{z}_h^{(a+\eta_1)} - \dot{z}_h^{(a+\eta_2)} \Vert_{H^1(\Omega)} \, \lesssim \, \Vert \alpha_1 - \alpha_2 \Vert_{C^1(\partial \Omega)} \,.
\end{align}
Furthermore, we can write
\begin{multline*}
(\dot{F}_{h,j}[\alpha_1] - \dot{F}_{h,j}[\alpha_2])\eta \, = \, \\
\int_{\partial \Omega}\phi_j (\dot{u}_h^{(a+\eta_1)}-\dot{u}_h^{(a+\eta_2)}) z_h^{(a+\eta_1)} \ds +  \int_{\partial \Omega}\phi_j \dot{u}_h^{(a+\eta_2)} (z_h^{(a+\eta_1)} - z_h^{(a+\eta_2)}) \ds  \\
+ \int_{\partial \Omega}\phi_j (u_h^{(a+\eta_1)}-u_h^{(a+\eta_2)}) \dot{z}_h^{(a+\eta_1)} \ds +  \int_{\partial \Omega}\phi_j u_h^{(a+\eta_2)} (\dot{z}_h^{(a+\eta_1)} - \dot{z}_h^{(a+\eta_2)}) \ds \, .
\end{multline*}
The Lipschitz continuity now follows by using the Cauchy--Schwarz inequality together with Lemma~\ref{lem:continuity}, Theorem~\ref{thm:frechz}, \eqref{eq:contdotu} and \eqref{eq:contdotz2}.
\end{proof}

Similarly, one can show consistency of $F_h$ with $F$ on $B_\rho(a)$, but in particular, a refined consistency at $\tilde{a}$ (see \cite[Def.\@ 3.3]{keller75}).
\begin{corollary}\label{cor:cons}
The family $F_h$ is consistent of order $2$ with $F$ at $\tilde{a}$.
\end{corollary}
\begin{proof}
We use Lemma~\ref{lem:zhbound} and obtain that
\begin{align*}
|F_{h,j}(\tilde{a})| \, \leq \, \Vert \phi_j \Vert_{C^1(\partial \Omega)} \Vert u_h^{(\tilde{a})} \Vert_{L^2(\partial \Omega)} \Vert z_h^{(\tilde{a})} \Vert_{L^2(\partial \Omega)} 
\, \lesssim \, h^2 \Vert \phi_j \Vert_{C^1(\partial \Omega)} \Vert u^{(\tilde{a})} \Vert_{H^2(\Omega)}^2\, .
\end{align*}
\end{proof}

\begin{lemma}\label{lem:normequ}
Suppose that 
$u^{(\tilde{a})}$ vanishes at most in isolated points on $\partial \Omega$.
Then, $\Vert a \Vert_* = \Vert u^{(\tilde{a})}a \Vert_{H^{-1/2}(\p \Omega)}$ is a norm on $C^1(\partial \Omega)$. 
\end{lemma}
\begin{proof}
To get a contradiction, suppose that $a \ne 0$ but $\Vert a \Vert_* = 0$.

Since $u^{(\tilde{a})} \in H^{3/2}(\partial \Omega) \subset C(\partial \Omega)$ and $a \in C^1(\partial \Omega)$, it holds that $u^{(\tilde{a})}a \in H^{1/2}(\partial \Omega)$ and 
$\Vert u^{(\tilde{a})}a \Vert_{H^{1/2}(\partial \Omega)} \lesssim \Vert u^{(\tilde{a})} \Vert_{H^{3/2}(\partial \Omega)} \Vert a \Vert_{C^1(\partial \Omega)}$.
Then
 \begin{align*}
\frac{|(u^{(\tilde{a})}a, u^{(\tilde{a})}a)_{L^2(\partial \Omega})|}{\Vert u^{(\tilde{a})}a \Vert_{H^{1/2}(\p \Omega)}}
\, \leq \, 
\sup_{\phi \ne 0} \frac{|(u^{(\tilde{a})}a, \phi)_{L^2(\partial \Omega)}|}{\Vert \phi \Vert_{H^{1/2}(\p \Omega)}}
\, = \, \Vert u^{(\tilde{a})}a \Vert_{H^{-1/2}(\p \Omega)} \, = \, 0\, .
    \end{align*}
We get $u^{(\tilde{a})}a = 0$, 
however, since $u^{(\tilde{a})}a$ is continuous and $u^{(\tilde{a})}$ vanishes at isolated points only, $a$ must vanish, a contradiction.
\end{proof}

\begin{remark}
Under the assumption of Lemma~\ref{lem:normequ}, $\Vert \cdot \Vert_*$ is a norm and on the finite dimensional space $V_J$ one can estimate $\Vert \cdot \Vert_{C^1(\partial \Omega)}$ against $\Vert \cdot \Vert_*$. 
It can be shown that the constant in the norm equivalence depends on the maximum of $\Vert \phi_i \Vert_{C^1(\partial \Omega)}$, $i=1,\dots,J$, but also on the terms $\max_{\xi \in V_J} \Vert \xi \Vert_{H^1(\Gamma)}/\Vert \xi \Vert_{L^2(\Gamma)}$ and $\max_{\xi \in V_J} \Vert \xi \Vert_{L^2(\Gamma)}/\Vert u^{(\tilde{a})}\xi \Vert_{L^2(\Gamma)}$. The first term will blow up as $J\to \infty$ regardless of $u^{{(\tilde{a})}}$.
Without a finite dimensionality constraint the problem is only logarithmically stable and in particular, the implicit constant in Theorem~\ref{thm:stability} blows up as $J \to \infty$.
\end{remark}

\begin{lemma}\label{lem:Fh}
Suppose that 
$u^{(\tilde{a})}$ vanishes at most at isolated points on $\partial \Omega$.
For a sufficiently small maximal mesh size $h$
the Fr\'echet derivative $\dot{F}_h[\tilde{a}]$ of $F_h$ at $\tilde{a}$ satisfies
    \begin{align}\label{eq:dotFhbound}
\Vert \eta \Vert_{C^1(\partial \Omega)}\, \lesssim \, \big| \dot{F}_h[\tilde{a}]\eta \big|_{\R^J} \, \quad \text{for all } \eta \in V_J\, \,.
    \end{align}
This implies that the inverse $(\dot{F}_h[\tilde{a}])^{-1}$ exists with a uniform bound in $0<h\leq h_0$ for some sufficiently small $h_0$.
The implicit constant in \eqref{eq:dotFhbound} depends on $\Vert \tilde{a} \Vert_{C^1(\partial \Omega)}$ and on $V_J$.
\end{lemma}
\begin{proof}
To each function $\eta \in V_J$ we associate a vector $\underline{\eta} \in \R^J$ with components $\eta_j$, $j=1,\dots, J$, for which holds
\begin{align}\label{eq:ne}
\eta \, = \, \sum_{j=1}^J \eta_j \phi_j \quad \text{with equivalent norms } \Vert \eta \Vert_{C^1(\partial \Omega)} \, \sim \, |\underline{\eta} |_{\R^J}\, .
\end{align}
We note that 
\begin{align}\label{eq:fhid}
\underline{\eta} \cdot \dot{F}_h[a] \eta \, = \, \int_{\partial \Omega} \eta (\dot{u}_h^{(a)} z_h^{(a)} + u_h^{(a)} \dot{z}_h^{(a)}) \ds \quad \text{for any } \eta \in V_J\, .
\end{align}
We consider the continuum counterpart to $\dot{z}_h^{(a)}$, which is 
\begin{subequations}
\begin{align}
\Delta \dot{z}^{(a)}\, &= \, 1_\omega \dot{u}_h^{(a)} \quad \text{in } \Omega \, ,\\
\partial_\nu \dot{z}^{(a)} + a\dot{z}^{(a)} \, &= \, -\eta z_h^{(a)} \quad \text{on } \partial \Omega \, . \label{eq:bcz}
\end{align}
\end{subequations}
Using Lemma~\ref{lem:fembound}, the $H^2$-regularity of Lemma~\ref{lem:H2reg} for $\dot{z}^{(\tilde{a})}$ as well as Lemma~\ref{lem:zhbound} yields
\begin{align}\label{eq:dotz-dotzh}
\begin{split}
\Vert \dot{z}^{(\tilde{a})} - \dot{z}_h^{(\tilde{a})} \Vert_{H^1(\Omega)} \, &\lesssim \, h \Vert \dot{z}^{(\tilde{a})} \Vert_{H^2(\Omega)} \\
\, &\lesssim \, h \big( \Vert \dot{u}_h^{(\tilde{a})} \Vert_{L^2(\omega)} + \Vert \eta \Vert_{C^1(\partial \Omega)} \Vert z_h^{(\tilde{a})} \Vert_{H^1(\Omega)} \big) \\
\, &\lesssim \, h \big( \Vert \dot{u}_h^{(\tilde{a})} \Vert_{L^2(\omega)} + h^2\Vert \eta \Vert_{C^1(\partial \Omega)} \big) \, .
\end{split}
\end{align}
In the same way one finds that
\begin{align}\label{eq:dotu-dotuh}
\Vert \dot{u}^{(\tilde{a})} - \dot{u}_h^{(\tilde{a})} \Vert_{H^1(\Omega)} \, \lesssim \, h\Vert \eta \Vert_{C^1(\partial \Omega)} \, . 
\end{align}
Using again that both functions $\dot{u}^{(\tilde{a})}$ and $\dot{z}^{(\tilde{a})}$ are in $H^2(\Omega)$
together with integration by parts yields
\begin{align}\label{eq:pidotuh}
\begin{split}
\Vert \dot{u}_h^{(\tilde{a})} \Vert_{L^2(\omega)}^2 \, &= \, \int_\Omega \dot{u}_h^{(\tilde{a})} \Delta \dot{z}^{(\tilde{a})} \dx  -\int_\Omega \Delta \dot{u}^{(\tilde{a})} \dot{z}_h^{(\tilde{a})} \dx   \\
&= \, - \int_{\Omega} \nabla \dot{u}_h^{(\tilde{a})} \cdot  \nabla \dot{z}^{(\tilde{a})} \dx + \int_{\Omega} \nabla \dot{u}^{(\tilde{a})} \cdot  \nabla \dot{z}_h^{(\tilde{a})} \dx   \\
&\phantom{=\, } + \int_{\partial \Omega} \dot{u}_h^{(\tilde{a})}  \partial_{\nu}\dot{z}^{(\tilde{a})} \ds  - \int_{\partial \Omega} \partial_{\nu}\dot{u}^{(\tilde{a})}  \dot{z}_h^{(\tilde{a})} \ds   \, .
\end{split}
\end{align}
We continue by estimating the volume integrals. We use the Cauchy--Schwarz inequality and see that 
\begin{multline*}
\left| -\int_{\Omega} \nabla \dot{u}_h^{(\tilde{a})} \cdot  \nabla \dot{z}^{(\tilde{a})} \dx + \int_{\Omega} \nabla \dot{u}^{(\tilde{a})} \cdot  \nabla \dot{z}_h^{(\tilde{a})} \dx   \right| \\
\, \leq \, \Vert \dot{u}_h^{(\tilde{a})} \Vert_{H^1(\Omega)} \Vert \dot{z}^{(\tilde{a})} - \dot{z}_h^{(\tilde{a})} \Vert_{H^1(\Omega)} + \Vert \dot{u}^{(\tilde{a})} - \dot{u}_h^{(\tilde{a})} \Vert_{H^1(\Omega)} \Vert \dot{z}_h^{(\tilde{a})} \Vert_{H^1(\Omega)} \, .
\end{multline*}
Next, we use that
for any $\varepsilon>0$ and any $a,b\in \R$ we have that $ab \leq 1/2(\varepsilon^2 a^2 + \varepsilon^{-2}b^2)$, the bounds \eqref{eq:dotz-dotzh} and \eqref{eq:dotu-dotuh}, Corollary~\ref{cor:dotuh} and Corollary~\ref{cor:dotzh} to see that
\begin{align}\label{eq:T1}
\begin{split}
\Vert \dot{u}_h^{(\tilde{a})} &\Vert_{H^1(\Omega)} \Vert \dot{z}^{(\tilde{a})} - \dot{z}_h^{(\tilde{a})} \Vert_{H^1(\Omega)} + \Vert \dot{u}^{(\tilde{a})} - \dot{u}_h^{(\tilde{a})} \Vert_{H^1(\Omega)} \Vert \dot{z}_h^{(\tilde{a})} \Vert_{H^1(\Omega)} \\
\, & \lesssim \, \varepsilon^2\big(\Vert \dot{u}_h^{(\tilde{a})} \Vert_{H^1(\Omega)}^2 + \Vert \dot{z}_h^{(\tilde{a})} \Vert_{H^1(\Omega)}^2\big) \\
&\phantom{\leq \, }+ \varepsilon^{-2}\big( \Vert \dot{z}^{(\tilde{a})} - \dot{z}_h^{(\tilde{a})} \Vert_{H^1(\Omega)}^2 + \Vert \dot{u}^{(\tilde{a})} - \dot{u}_h^{(\tilde{a})} \Vert_{H^1(\Omega)}^2 \big) \\
\, & \lesssim \, \varepsilon^2 \Vert \dot{u}_h^{(\tilde{a})} \Vert_{L^2(\omega)}^2 + \varepsilon^{-2}h^2 \Vert \dot{u}_h^{(\tilde{a})} \Vert_{L^2(\omega)}^2 + h^2 \Vert \eta \Vert_{C^1(\partial \Omega)}^2 \, .
\end{split}
\end{align}
Here, we do not emphasize the $\varepsilon^{-2}$ contribution in the last term.
We return to \eqref{eq:pidotuh} use the previous estimate and rearrange for a sufficiently small constant $\varepsilon$ and mesh size $h$ to obtain that
\begin{align}\label{eq:newdotubound}
\Vert \dot{u}_h^{(\tilde{a})} \Vert_{L^2(\omega)}^2 \, \lesssim \, \left|  \int_{\partial \Omega} \dot{u}_h^{(\tilde{a})}  \partial_{\nu}\dot{z}^{(\tilde{a})} \ds -\int_{\partial \Omega} \partial_{\nu}\dot{u}^{(\tilde{a})}  \dot{z}_h^{(\tilde{a})} \ds   \right| 
+ h^2 \Vert \eta \Vert_{C^1(\partial \Omega)}^2  \, .
\end{align}
We apply the boundary conditions in \eqref{eq:bcu} and \eqref{eq:bcz}, use the representation in \eqref{eq:fhid} and use the triangle inequality to see that
\begin{align*}
& \left|\int_{\partial \Omega} \dot{u}_h^{(\tilde{a})}  \partial_{\nu}\dot{z}^{(\tilde{a})} \ds - \int_{\partial \Omega} \partial_{\nu}\dot{u}^{(\tilde{a})}  \dot{z}_h^{(\tilde{a})} \ds   \right| \\
&= \, \left|\int_{\partial \Omega} \tilde{a} \dot{u}^{(\tilde{a})} \dot{z}_h^{(\tilde{a})}\ds - \int_{\partial \Omega} \tilde{a} \dot{u}_h^{(\tilde{a})} \dot{z}^{(\tilde{a})} \ds  - 2 \int_{\partial \Omega} \eta \dot{u}_h^{(\tilde{a})} z_h^{(\tilde{a})} \ds + \underline{\eta} \cdot \dot{F}_h[a] \eta\right| \\
&\leq \, \left|\int_{\partial \Omega} \tilde{a} \dot{u}^{(\tilde{a})} \dot{z}_h^{(\tilde{a})}\ds  - \int_{\partial \Omega} \tilde{a} \dot{u}_h^{(\tilde{a})} \dot{z}^{(\tilde{a})} \ds\right| + 2 \left|\int_{\partial \Omega} \eta \dot{u}_h^{(\tilde{a})} z_h^{(\tilde{a})} \ds \right| 
+ \big|\underline{\eta} \cdot \dot{F}_h[\tilde{a}] \eta \big| \\
&=: \, T_{1}(\tilde{a})  + T_{2}(\tilde{a})  + \big|\underline{\eta} \cdot \dot{F}_h[\tilde{a}] \eta \big| \, .
\end{align*}
The trace estimate and \eqref{eq:T1} show that
\begin{align*}
T_{1}(\tilde{a}) 
\, &\lesssim \, \big( \Vert \dot{u}_h^{(\tilde{a})} \Vert_{L^2(\partial \Omega)} \Vert \dot{z}^{(\tilde{a})} - \dot{z}_h^{(\tilde{a})} \Vert_{L^2(\partial \Omega)} + \Vert \dot{z}_h^{(\tilde{a})} \Vert_{L^2(\partial \Omega)} \Vert \dot{u}^{(\tilde{a})} - \dot{u}_h^{(\tilde{a})} \Vert_{L^2(\partial \Omega)} \big) \\
\, & \lesssim \, \varepsilon^2 \Vert \dot{u}_h^{(\tilde{a})} \Vert_{L^2(\omega)}^2 + \varepsilon^{-2}h^2 \Vert \dot{u}_h^{(\tilde{a})} \Vert_{L^2(\omega)}^2 + h^2 \Vert \eta \Vert_{C^1(\partial \Omega)}^2 \, .
\end{align*}
Furthermore, by Corollary~\ref{cor:dotuh} and Lemma~\ref{lem:zhbound}
\begin{align*}
T_2(\tilde{a}) \, &\lesssim \, \Vert \eta \Vert_{C^1(\partial \Omega)} \Vert \dot{u}_h^{(\tilde{a})} \Vert_{H^1(\Omega)} \Vert z_h^{(\tilde{a})} \Vert_{H^1(\Omega)} \\
\, &\lesssim \, h^2 \Vert \eta \Vert_{C^1(\partial \Omega)}\Vert \dot{u}_h^{(\tilde{a})} \Vert_{L^2(\omega)} + h^3 \Vert \eta \Vert_{C^1(\partial \Omega)}^2 \, .
\end{align*}
Combining the bounds for $T_1(\tilde{a})$ and $T_2(\tilde{a})$ and rearranging \eqref{eq:newdotubound} now shows that for small enough $h$
\begin{align*}
\Vert \dot{u}_h^{(\tilde{a})} \Vert_{L^2(\omega)}^2 \, \lesssim \, \big|\underline{\eta} \cdot \dot{F}_h[\tilde{a}] \eta \big|
+ h^2 \Vert \eta \Vert_{C^1(\partial \Omega)}^2 \, 
\end{align*}
and a combination with Corollary~\ref{cor:dotuh} yields
\begin{align*}
\Vert \dot{u}_h^{(\tilde{a})} \Vert_{H^1(\Omega)}^2 \, \lesssim \, \big|\underline{\eta} \cdot \dot{F}_h[\tilde{a}] \eta \big|
+ h^2 \Vert \eta \Vert_{C^1(\partial \Omega)}^2 \, .
\end{align*}
For $\eta \in V_J$ we now use Lemma~\ref{lem:normequ} and Lemma~\ref{lem:fembound} and get that
\begin{align*}
\Vert \eta \Vert_{C^1(\partial \Omega)}^2 \, \lesssim \, \Vert \eta \Vert_{*}^2 
\, &\lesssim \, \Vert \eta( u^{(\tilde{a})} - u_h^{(\tilde{a})})\Vert_{H^{-1/2}(\partial \Omega)}^2 + \Vert \eta u_h^{(\tilde{a})}\Vert_{H^{-1/2}(\partial \Omega)}^2 \\
& \lesssim \, \Vert \eta u_h^{(\tilde{a})}\Vert_{H^{-1/2}(\partial \Omega)}^2 +  h^2 \Vert \eta \Vert_{C^1(\partial \Omega)}^2 \, .
\end{align*}
Now we apply the boundary condition in \eqref{eq:bcu} as well as the fact that 
\begin{align*}
\Vert \partial_\nu\dot{u}^{(\tilde{a})} \Vert_{H^{-1/2}(\partial \Omega)} \lesssim \Vert \dot{u}^{(\tilde{a})} \Vert_{H^{1}(\Omega)}
\end{align*}
(see e.g. \cite[Lem.\@ 4.3]{McLean00}), which gives that
\begin{align*}
\Vert \eta u_h^{(\tilde{a})}\Vert_{H^{-1/2}(\partial \Omega)}^2 \, = \, \Vert \partial_\nu \dot{u}^{(\tilde{a})} + \tilde{a}\dot{u}^{(\tilde{a})}\Vert_{H^{-1/2}(\partial \Omega)}^2 \, \lesssim \, \Vert \dot{u}^{(\tilde{a})} \Vert_{H^1(\Omega)}^2 
\end{align*}
and by \eqref{eq:dotu-dotuh},
\begin{align*}
\Vert \dot{u}^{(\tilde{a})} \Vert_{H^1(\Omega)}^2  \, \lesssim \, \Vert \dot{u}_h^{(\tilde{a})} \Vert_{H^1(\Omega)}^2 + h^2 \Vert \eta \Vert_{C^1(\partial \Omega)}^2 \, .
\end{align*}
Combining the last inequalities finally yields
\begin{align*}
\Vert \eta \Vert_{C^1(\partial \Omega)}^2 \, \lesssim \, \big|\underline{\eta} \cdot \dot{F}_h[\tilde{a}] \eta \big|
+ h^2 \Vert \eta \Vert_{C^1(\partial \Omega)}^2\, .
\end{align*}
For small enough $h$, the last term can be absorbed into the left side. Finally, the Cauchy-Schwarz inequality in $\R^J$ together with the norm equivalence in \eqref{eq:ne} yields
\begin{align*}
\Vert \eta \Vert_{C^1(\partial \Omega)}\, \lesssim \, \big| \dot{F}_h[\tilde{a}]\eta \big|_{\R^J}\, .
    \end{align*}
\end{proof}
\begin{remark}
Theorem~\ref{thm:main} guarantees convergence of the Newton method for initial guesses, which are sufficiently close to $a_h$.
This is a purely local convergence property.
Choosing $a = \zeta$ with $\zeta>0$ in the Robin problem \eqref{eq:robin-gen} and letting $\zeta \to \infty$ turns the Robin boundary condition into a homogeneous Dirichlet condition, i.e., $u=0$ on $\partial \Omega$. With this property, the function $F$ from \eqref{eq:F} vanishes. This shows that global convergence cannot be expected from the Newton method. 
\end{remark}
\section{Perturbation Analysis}\label{sec:pert}
In this short section we study the situation, in which $q = u^{(\tilde{a})}|_\omega$ is perturbed by some noise.
For this purpose, let $q^\delta = q + \delta$, where $\delta \in L^2(\omega)$ is a perturbation with $\Vert \delta \Vert_{L^2(\omega)}$ sufficiently small.
Such a perturbation affects $z_h^{(a)}$, the solution of \eqref{eq:zh}. We denote by $z_h^{(a),\delta}$ the solution of \eqref{eq:zh}, when $q$ is replaced by $q^\delta$.
Using Lemma~\ref{lem:zhbound} and the well-posedness of the problem \eqref{eq:zh} yields that
\begin{align*}
\Vert z_h^{(\tilde{a}),\delta} \Vert_{H^1(\Omega)} 
\, \lesssim \, h^2 \Vert u^{(\tilde{a})} \Vert_{H^2(\Omega)} + \Vert \delta \Vert_{L^2(\omega)}\, .
\end{align*}
We define a function $F_h^\delta$ that involves noisy data by
\begin{align}\label{eq:Fhdelta}
F_h^\delta (a) \, = \,  (F_{h,1}^\delta(a), \dots, F_{h,J}^\delta(a) )\, , \; \; \text{where } F_{h,j}^\delta(a) \, = \, \int_{\partial \Omega} \phi_j u_h^{(a)} z_h^{(a),\delta} \ds\, \, .
\end{align}

Using the very same framework as in Theorem~\ref{thm:frechz}, the function $\dot{z}_h^{(a),\delta}$, the Gateaux derivative of $z_h^{(a),\delta}$ with respect to the Robin parameter, can be established.
With this function, the Fr\'echet derivative $\dot{F}_h^\delta[a]$ of \eqref{eq:Fhdelta} is given as in \eqref{eq:dotFh} with $z_h^{(a)}$ and $\dot{z}_h^{(a)}$ replaced by $z_h^{(a),\delta}$ and $\dot{z}_h^{(a),\delta}$, respectively.
We see that Corollary~\ref{cor:contF} still holds when $\dot{F}_h$ is replaced by $\dot{F}_h^\delta$, while \eqref{eq:F_hcons2} needs to be replaced by
\begin{align}\label{eq:Fhdeltabound}
\left|F_h^\delta (\tilde{a})\right|_{\R^J} \, \lesssim \, h^2 + \Vert \delta \Vert_{L^2(\omega)}\, .
\end{align}
We can show the following Lemma, which is similar to the unperturbed case in Lemma~\eqref{lem:Fh}.
\begin{lemma}\label{lem:Fhdelta}
Suppose that
$u^{(\tilde{a})}$ vanishes at most at isolated points on $\partial \Omega$
and that $\Vert \delta \Vert_{L^2(\omega)}$ is sufficiently small. For a sufficiently small maximal mesh size $h$
the Fr\'echet derivative $\dot{F}_h^\delta[\tilde{a}]$ of $F_h^\delta$ at $\tilde{a}$ satisfies
    \begin{align}\label{eq:dotFhdeltabound}
\Vert \eta \Vert_{C^1(\partial \Omega)}\, \lesssim \, \big| \dot{F}_h^\delta[\tilde{a}]\eta \big|_{\R^J} \, \quad \text{for all } \eta \in V_J\, \,.
    \end{align}
This implies that the inverse $(\dot{F}_h^\delta[\tilde{a}])^{-1}$ exists with a uniform bound in $0<h\leq h_0$ for some sufficiently small $h_0$.
The implicit constant in \eqref{eq:dotFhdeltabound} depends on $\Vert \tilde{a} \Vert_{C^1(\partial \Omega)}$ and on $V_J$.
\end{lemma}
\begin{proof}
The proof can be done as the proof of Lemma~\ref{lem:Fh} with some small modifications. We require the estimate
\begin{align*}
&\Vert \dot{z}^{(\tilde{a})} - \dot{z}_h^{(\tilde{a}),\delta} \Vert_{H^1(\Omega)} \\
\, &\lesssim \, \Vert \dot{z}^{(\tilde{a})} - \dot{z}_h^{(\tilde{a})} \Vert_{H^1(\Omega)} + \Vert \dot{z}_h^{(\tilde{a})}- \dot{z}_h^{(\tilde{a}),\delta}  \Vert_{H^1(\Omega)}  \\
& \lesssim \, 
h \Vert \dot{u}_h^{(\tilde{a})} \Vert_{L^2(\omega)} + \Vert \eta \Vert_{C^1(\partial \Omega)} \Vert z_h^{(\tilde{a})} - z_h^{(\tilde{a}),\delta} \Vert_{H^1(\Omega)} 
+  h^3\Vert \eta \Vert_{C^1(\partial \Omega)}  \\
& \lesssim \, h \Vert \dot{u}_h^{(\tilde{a})} \Vert_{L^2(\omega)} + h^3\Vert \eta \Vert_{C^1(\partial \Omega)} + \Vert \eta \Vert_{C^1(\partial \Omega)} \Vert \delta \Vert_{L^2(\omega)} \, ,
\end{align*} 
that we deduced by using the bound from \eqref{eq:dotz-dotzh} as well as the well-posedness estimate of $\dot{z}_h^{(\tilde{a})}- \dot{z}_h^{(\tilde{a}),\delta}$.
Finally, all the rearrangements that appear in the proof of Lemma~\ref{lem:Fh} can still be done under the assumption that $\Vert \delta \Vert_{L^2(\omega)}$ is small enough.
\end{proof}
The existence of the element $a_h \in B_\rho(\tilde{a})$ from Theorem~\ref{thm:main} is guaranteed by \cite[Thm.\@ 3.6]{keller75}, which uses the Banach fixed point theorem and the consistency from Corollary~\ref{cor:cons}. 
Under the assumption that $\Vert \delta \Vert_{L^2(\omega)}$ is sufficiently small, we can achieve a similar result that we summarize in the next theorem.
\begin{theorem}\label{thm:mainN}
Let $\tilde{a} \in V_J$ and let $u^{(\tilde{a})} \in H^1(\Omega)$ denote the solution to \eqref{eq:robin-gen-ex} for $a=\tilde{a}$ and for some $f\in L^2(\Omega)$ and $g\in H^{1/2}(\partial \Omega)$. Moreover, let $\delta \in L^2(\omega)$ with $\Vert \delta \Vert_{L^2(\omega)}$ sufficiently small.
Assume that
$u^{(\tilde{a})}$ vanishes at most at isolated points
on $\partial \Omega$. Moreover, let $q^\delta=u^{(\tilde{a})}|_\omega + \delta$ for some $\omega \subset \Omega$ and $f$ and $g$ be known. 
Then, 
\begin{itemize}
\item[(a)] For some sufficiently small $0<h \leq h_0$ and $\rho>0$, there is a unique $a_h^\delta \in B_\rho(\tilde{a})$ such that $F_h^\delta(a_h^\delta) = 0$. Furthermore,
\begin{align}\label{eq:h2boundpert}
\Vert \tilde{a} - a_h^\delta \Vert_{C^1(\partial \Omega)} \, \lesssim \, h^2 + \Vert \delta \Vert_{L^2(\omega)}\, .
\end{align}
The implicit constant in \eqref{eq:h2boundpert} depends on $\Vert \tilde{a} \Vert_{C^1(\partial \Omega)}$ and on $V_J$.
\item[(b)] For an initial guess that is sufficiently close to the unique root $a_h^\delta \in B_\rho(\tilde{a})$ from part (a) the Newton iterates corresponding to the equation $F_h^\delta(x) = 0$ converge quadratically to $a_h^\delta$.
\end{itemize}
\end{theorem}
\begin{proof}
We follow the proofs of \cite[Thm.\@ 3.4, Thm.\@ 3.6]{keller75}.
We define the mapping $G(a) = a - \dot{F}_h^\delta[\tilde{a}]^{-1}F_h^\delta(a)$ and note that for a sufficiently small $\rho>0$ and for $a_1,a_2 \in B_\rho(\tilde{a})$ there holds that
\begin{align*}
G(a_1) - G(a_2) \, = \, \dot{F}_h^\delta[\tilde{a}]^{-1}\left(\dot{F}_h^\delta[\tilde{a}] - \check{F}_h^\delta(a_1,a_2) \right)(a_1-a_2)\, ,
\end{align*}
where 
\begin{align*}
\check{F}_h^\delta(a_1,a_2) \, = \, \int_0^1 \dot{F}_h^\delta[ta_1 + (1-t)a_2)] \dif t \, .
\end{align*}
One can now use the Lipschitz continuity of $\dot{F}_h^\delta$ as well as Lemma~\ref{lem:Fhdelta} to see that 
\begin{align}\label{eq:Glip}
\Vert G(a_1) - G(a_2)\Vert_{C^1(\partial \Omega)} \, \leq \, C \rho \Vert a_1 - a_2 \Vert_{C^1(\partial \Omega)}\, . 
\end{align}
for some $C>0$. Assuming that $\rho$ is small enough yields that $G$ is a contraction.
Moreover, by using \eqref{eq:Fhdeltabound} we see that
\begin{align*}
\Vert \tilde{a} - G(\tilde{a}) \Vert_{C^1(\partial \Omega)} \, \lesssim \,  |F_h^\delta(\tilde{a})|_{\R^J} \, \lesssim \,  h^2 + \Vert \delta \Vert_{L^2(\omega)}\, .
\end{align*}
We now need to assume that both $h$ and $\Vert \delta\Vert_{L^2(\omega)}$ are so small that
\begin{align}\label{eq:acont}
\Vert \tilde{a} - G(\tilde{a}) \Vert_{C^1(\partial \Omega)} \, \leq \,  (1-C\rho) \rho \, ,
\end{align}
where $C>0$ is as in \eqref{eq:Glip}. In this case $G: B_\rho(\tilde{a}) \to B_\rho(\tilde{a})$, since for any $a\in B_\rho(\tilde{a})$, by using \eqref{eq:Glip} and \eqref{eq:acont} we get that
\begin{align*}
\Vert \tilde{a} - G(a) \Vert_{C^1(\partial \Omega)} \, \leq \, \Vert \tilde{a} - G(\tilde{a}) \Vert_{C^1(\partial \Omega)} + \Vert G(\tilde{a}) - G(a) \Vert_{C^1(\partial \Omega)} \, \leq \, \rho\, .
\end{align*}
Part (b) can be shown as \cite[Thm.\@ 3.7]{keller75} by additionally using similar considerations as before.
\end{proof}

\section{Numerical examples}
In this section, the aim is to numerically reconstruct the Robin function $\tilde{a}$ from the given $q = u^{(\tilde{a})}|_{\omega}$ (or likewise from $q^\delta = q+\delta$), where $\omega \subset \Omega$ and $u$ solves \eqref{eq:robin-gen-ex}. 
Throughout this section we set $\Omega = [0,1]^2$ and denote by $p : [0,4] \to \partial \Omega$ the positively oriented curve that parametrizes $\partial \Omega$ by arc length starting in the bottom left corner of the domain.
We assume that the function $\tilde{a}$ with $0 < a_0 \leq \tilde{a}$ lies in the $J = J_1+J_2$ dimensional space $V_{J_1,J_2} \subset C^1(\partial \Omega)$, which is the span of trigonometric functions 
\begin{align}\label{eq:phifun}
\phi_{1,m-1}(t) \, = \, \frac{1}{2} \cos\left( (m-1) \frac{\pi}{2}t\right)\, ,\quad \phi_{2,n}(t) \, = \, \frac{1}{2} \sin\left( n \frac{\pi}{2}t\right)\, ,
\end{align}
where $m=1,\dots,J_1$ and $n=1,\dots, J_2$ and $t \in [0,4]$ satisfies $p(t) = (x,y) \in \partial \Omega$.
Any function $a \in V_{J_1,J_2} $ can thus be written as
\begin{align}\label{eq:atilde}
a \, = \, \sum_{m=1}^{J_1} \alpha_m \phi_{1,m-1} + \sum_{n=1}^{J_2} \beta_n \phi_{2,n}\, ,
\end{align}
or in other words, $a \in V_{J_1,J_2} $ can be uniquely associated to vectors $\alpha\in \R^{J_1}$ and $\beta\in \R^{J_2}$ with entries $\alpha_{m}$, $m=1,\dots, J_1$ and $\beta_{n}$, $n = 1, \dots, J_2$, respectively.
Our approach in reconstructing $\tilde{a}$ is to simultaneously find the $J$ zeros of $F_{h,j}$ from \eqref{eq:Fh} by applying Newton's method to $F_h$. 
In $F_{h,j}$ from \eqref{eq:Fh}, $\phi_j$ is replaced by $\phi_{1,m-1}$ and $\phi_{2,n}$ from \eqref{eq:phifun}. 
Recall that, by Theorem~\ref{thm:main}, for a sufficiently close initial guess to $\tilde{a}$, Newton's method does converge to a root of $F_h$ quadratically.
It turns out that, when using noisy data, additional stability in the reconstruction can be achieved by using a simple backtracking line search. 
We incorporate this in the following way.
Denote by $x_k\in \R^J$ the $k$th iterate of the root search and by $a(x_k) \in V_{J_1,J_2} $ the linear combination of the functions in \eqref{eq:phifun} with individual contributions of each function determined by the entries of the vector $x_k \in \R^J$.
Once we determined the update $d_k \in \R^J$ by using Newton's method, we find the smallest number $\kappa \in \N_0$ such that
\begin{align}\label{eq:linesearchcrit}
\left|F_h\left(a\left(x_k + (0.5)^\kappa d_k\right)\right)\right| \, \leq \,  \left|F_h\left(a(x_k\right))\right| \quad \text{and } \quad a\left(x_k + (0.5)^\kappa d_k\right) > 0
\end{align} 
and define the next iterate to be $x_{k+1} = x_k + (0.5)^\kappa d_k$.
For the convenience of the reader we provide a detailed description of the reconstruction process including the computation of the Jacobi matrix in Newton's method in Algorithm~\ref{algo:newton}.
\begin{algorithm}[t]\label{algo:newton}
\DontPrintSemicolon
  
  \KwData{$f\in L^2(\Omega)$, $g \in H^{1/2}(\partial \Omega)$, $q\in L^2(\Omega)$, $\omega \subset \Omega$, $\phi_j$, $j=1,\dots, J$}
  \KwInput{$N, J_1, J_2 \in \N,$ $\alpha^{(0)} \in \R^{J_1}$, $\beta^{(0)} \in \R^{J_2}$, $\mathrm{tol}>0$}
  \KwOutput{$a_h$}

 Discretize the domain $\Omega$ into $(N+1)^2$ equidistant points. Then, $h=\sqrt{2}/N$\;
 $x_0 \gets \text{concatenate}(\alpha^{(0)},\, \beta^{(0)}) \in \R^{J}$ and $x_{-1} \gets 0 \in \R^J$\; 
 Def.\@ $a$ with entries $\alpha_{m}$ of $\alpha^{(0)}$ and entries $\beta_n$ of $\beta^{(0)}$ as in \eqref{eq:atilde} \;
 $k \gets 0 \in \N_0$ \;
 \While{$|x_k - x_{k-1}|_{\R^J}>\mathrm{tol}$}
   {
   Comp.\@ approximations $u_h^{(a)} \in V_h$ as in \eqref{eq:feform} and $z_h^{(a)} \in V_h$ as in \eqref{eq:zh} \label{algstep1} \;
  Comp.\@ $F_h(a)$ from \eqref{eq:Fh} with $u_h^{(a)}$ and $z_h^{(a)}$ from step \ref{algstep1} and $\phi_j$ from \eqref{eq:phifun}\;
   \For{$j=1,\dots,J$}    
        { 
       	Def.\@ $\eta$ as a linear combination of $\phi_{1,m-1}$ and $\phi_{2,n}$ from \eqref{eq:phifun} as in \eqref{eq:atilde} with $\alpha_i=\delta_{i,j}$, $\beta_i=0$ \textbf{if} $j\leq J_1$ and $\alpha_i=0$, $\beta_i=\delta_{i,j-J_1}$ \textbf{else} \label{algstep2}\;
       	Comp.\@ approximations $\dot{u}_h^{(a)} \in V_h$ as in \eqref{eq:uhdot} and $\dot{z}_h^{(a)} \in V_h$ as in \eqref{eq:zhdot} 
       	using $\eta$ from step \ref{algstep2} and $u_h^{(a)}$ and $z_h^{(a)}$ from step \ref{algstep1} \label{algstep3}\;
       	Comp.\@ $\dot{F}_h[a]\eta$ from \eqref{eq:dotFh} with $u_h^{(a)}$ and $z_h^{(a)}$ from step \ref{algstep1}, $\phi_j$ from \eqref{eq:phifun}, $\eta$ from step \ref{algstep2} and $\dot{u}_h^{(a)}$ and $\dot{z}_h^{(a)}$ from step \ref{algstep3} \;
       	Def.\@ the $j$th column of $J_{F_h}$ as $\dot{F}_h[a]\eta$ \label{algstep4}
        }
    Comp.\@ the Newton update $d_k = -J_{F_h}^{-1}F_h(a)$\;
    Find the smallest integer $\kappa$ such that \eqref{eq:linesearchcrit} is satisfied \;
   $x_{k+1} \gets x_k + (0.5)^\kappa d_k$ \;
   Def.\@ $\alpha^{(k+1)}$ as the first $J_1$ and $\beta^{(k+1)}$ as the last $J_2$ entries of $x_{k+1}$\;
   Def.\@ $a$ with entries $\alpha_{m}$ of $\alpha^{(k+1)}$ and entries $\beta_n$ of $\beta^{(k+1)}$ as in \eqref{eq:atilde} \;
   $k \gets k+1$
   }
$a_h \gets a$
\caption{Reconstruction of the Robin parameter $\tilde{a}$ with $\Omega = [0,1]^2$}
\end{algorithm}

Finite element simulations were performed using the open-source computing platform FEniCS (see \cite{AlnaesEtal14, Dolfinx23, BasixJoss, ScroggsEtal2022}).
For simulating $q$ we apply the same finite element solver that we use later on to evaluate both $F_h$ as well as its Jacobi matrix with entries $\dot{F}_{h,j}$ from \eqref{eq:dotFh}.
However, we use a higher degree of polynomials and a finer mesh size for this computation and interpolate the finite element solution afterwards to the coarser grid, on which we perform the reconstruction. This is supposed to mitigate the risk of inverse crime.

For all examples we define the right hand side in \eqref{eq:robin-gen-ex} as
\begin{align*}
f(x,y) \, =\,  -10x\mathrm{e}^{\sin(4\pi y)}\, \quad \text{for } (x,y) \in \Omega \quad \text{and} \quad g \, = \, 0 \, \quad \text{on } \partial \Omega\, .
\end{align*}
Moreover, let the exact $\tilde{a}$ be defined by a linear combination of $\phi_{1,m-1}$ and $\phi_{2,n}$ from \eqref{eq:phifun} as in \eqref{eq:atilde} with vectors
\begin{align}\label{eq:alphabeta}
\alpha \, = \, [10,\, 1,\,-0.5,\,2,\, 1,\, -0.5]^\top \,  \quad \text{and } \quad
\beta \, = \, [0.2,\, 1,\, -0.5,\, 2,\, 1,\, -0.5]^\top \, .
\end{align}
We assume the domain $\omega$ to be the union of the two discs around the points $[0.8, \,0.8]^\top$ and $[0.4,\,0.2]^\top$ with radii $0.05$ and $0.1$.
This corresponds to a known subdomain with an approximate volume of $4\%$ of the domain $\Omega$.
For our finite element simulations we pick the finest mesh size $h\approx 7\times 10^{-4}$ to compute $q$.
The solution $u^{(\tilde{a})}$ on $\Omega$, as well as $q$ on $\omega \subset \Omega$ are depicted in the top left and top middle plot of Figure~\ref{fig:ex1}, respectively. 
We find (visually and by Remark~\ref{rem:fandgass}) that $u^{(\tilde{a})} >0$ on $\partial \Omega$, which fulfills the requirements of Theorem~\ref{thm:main}.

\textbf{Example 1.}
In our first numerical example we study the convergence of Newton's method visually.
For the reconstruction scheme, let $h\approx 3\times 10^{-3}$ be the maximal mesh size.
\begin{figure}[t!]
\centering 
\includegraphics[scale=.28]{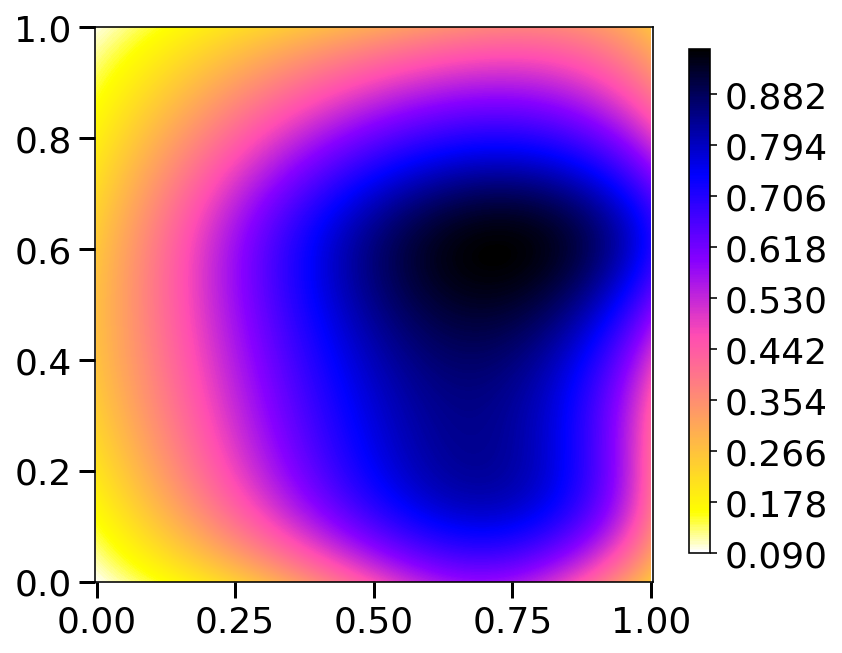}\hfill
\includegraphics[scale=.28]{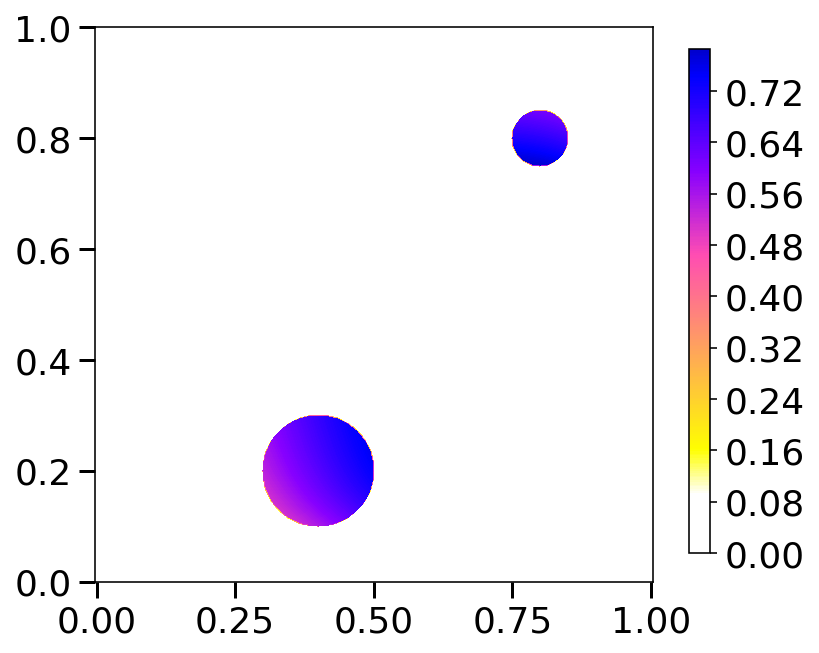}\hfill
\includegraphics[scale=.24]{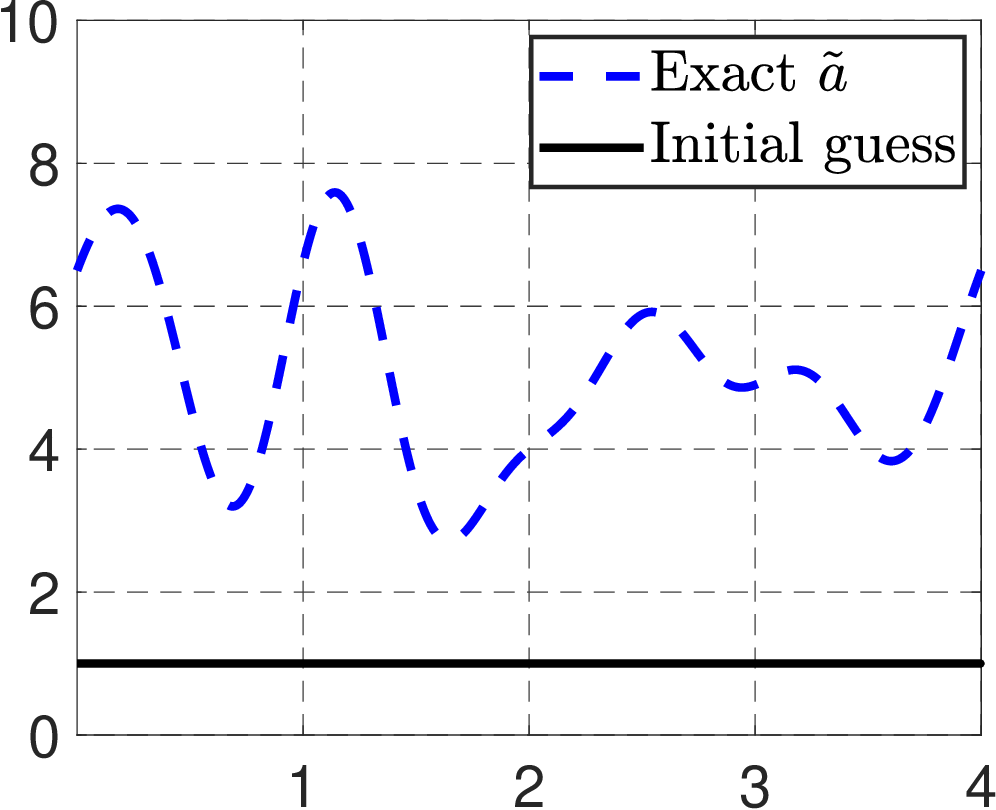}\hfill \\
\includegraphics[scale=.24]{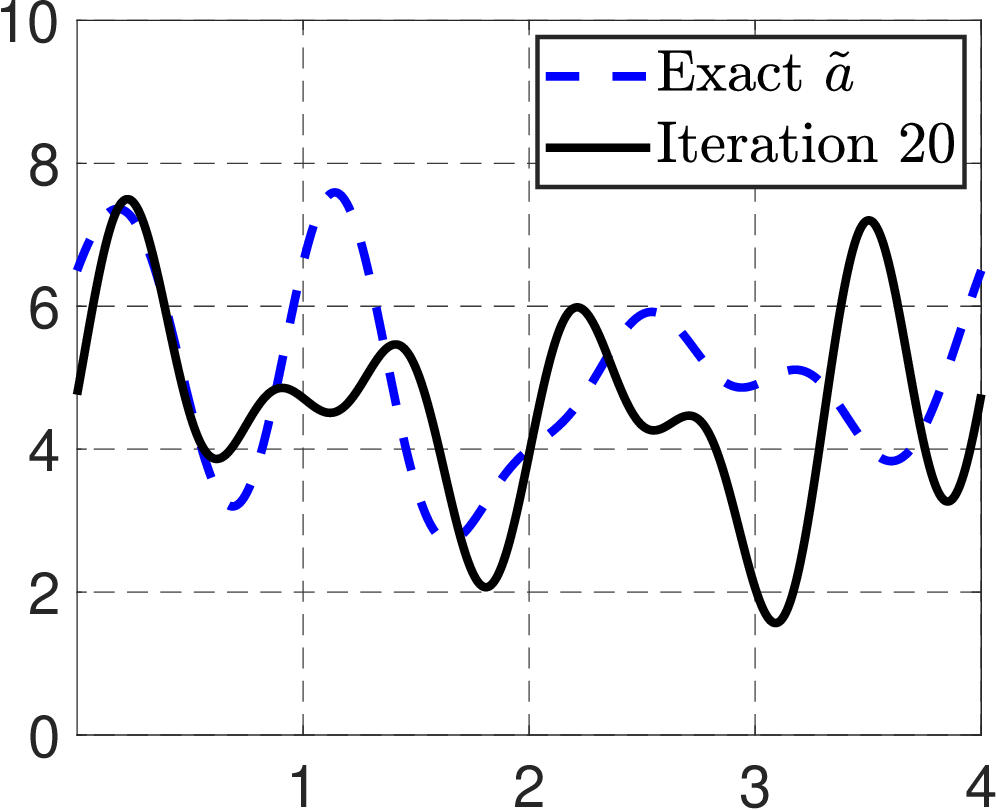}\hfill
\includegraphics[scale=.24]{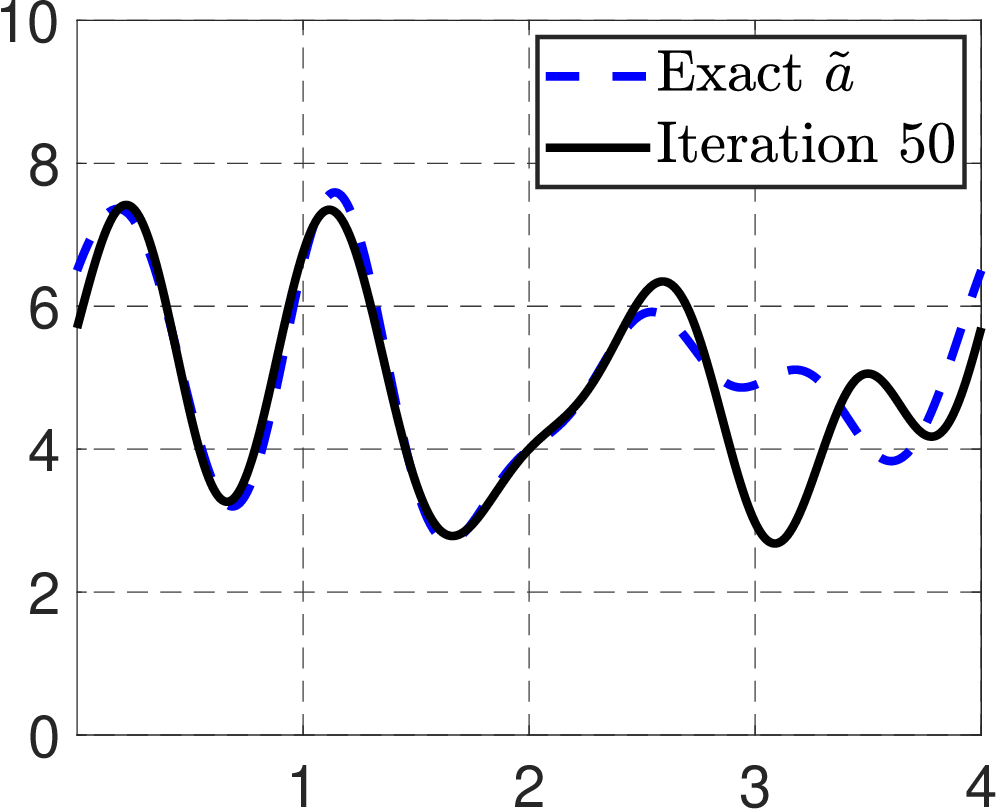}\hfill
\includegraphics[scale=.24]{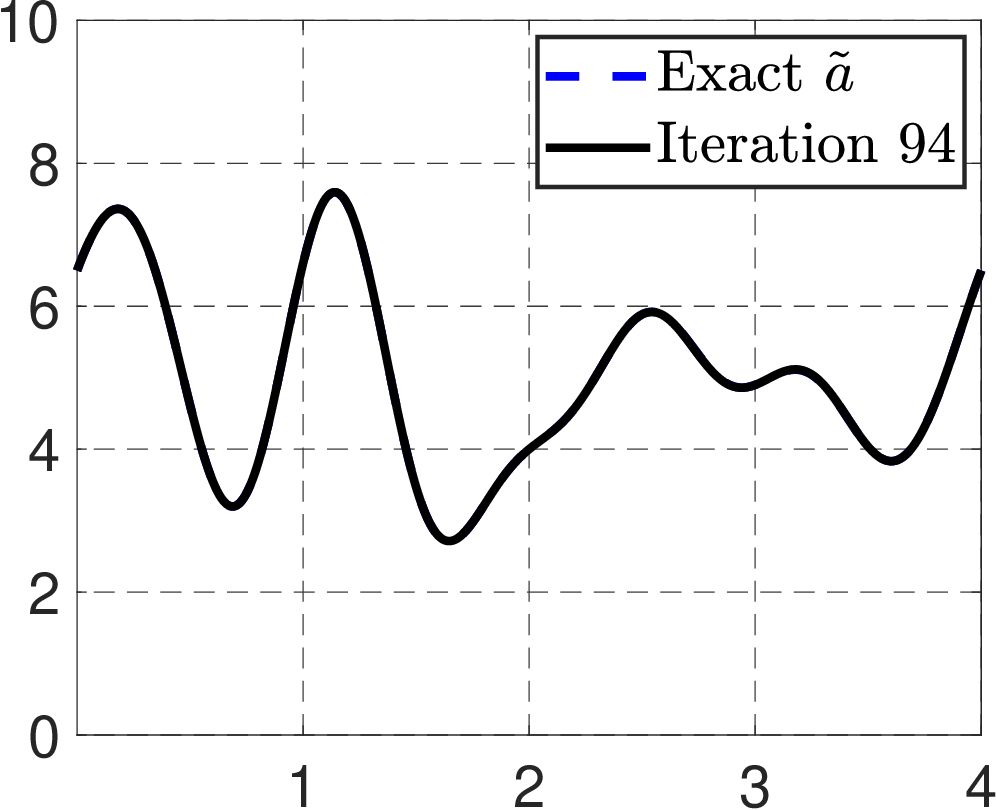}\hfill
\caption{Top left: Visualization of the exact solution on the domain $\Omega = [0,1]^2$. Top middle: Visualization of $q$ on the domain $\omega$. Top right to bottom right: The initial guess, the iterates 20, 50 and the final iteration 94. }
\label{fig:ex1}
\end{figure}
We start our reconstruction with a constant initial guess $a = 1$ as seen in the top right plot of Figure~\ref{fig:ex1}.
The bottom row in Figure~\ref{fig:ex1} displays some snapshot of the convergence history. The iteration stops after 94 steps with a flawless reconstruction of the true function $\tilde{a}$ as seen in the bottom right plot of Figure~\ref{fig:ex1}.
We emphasize that a successful reconstruction of $\tilde{a}$ depends on the initial guess of the Newton scheme, the provided data $q$ (in particular on the given set $\omega$) and on $\tilde{a}$ itself. 
In this example, e.g., if one disc of $\omega$ is removed, then the algorithm fails to converge to a reasonable approximation of $\tilde{a}$.

\textbf{Example 2.}
The purpose of this example is to verify \eqref{eq:h2bound} numerically.
We compute several reconstructions for different mesh sizes ranging from approximately $10^{-1}$ to $10^{-3}$.
For each simulation we start our reconstruction with the initial guess $a=1$. 
Subsequently we consider the final iterate of this simulation, denoted by $a_h$, and compute the error $\Vert \tilde{a} - a_h \Vert_{C^1(\partial \Omega)}$.
The corresponding double logarithmic plot showing the mesh size against the error is found in Figure~\ref{fig:ex2}.
Moreover, we provide in Figure~\ref{fig:ex2} a table that includes some mesh sizes, together with their corresponding error and the estimated order of convergence (EOC). Note that the values of the mesh size and the error are rounded values here.
\begin{figure}[t!]
\centering 
\begin{minipage}[t]{0.5\textwidth}
\adjustbox{valign=t,height=.27cm}{\includegraphics[scale=.38]{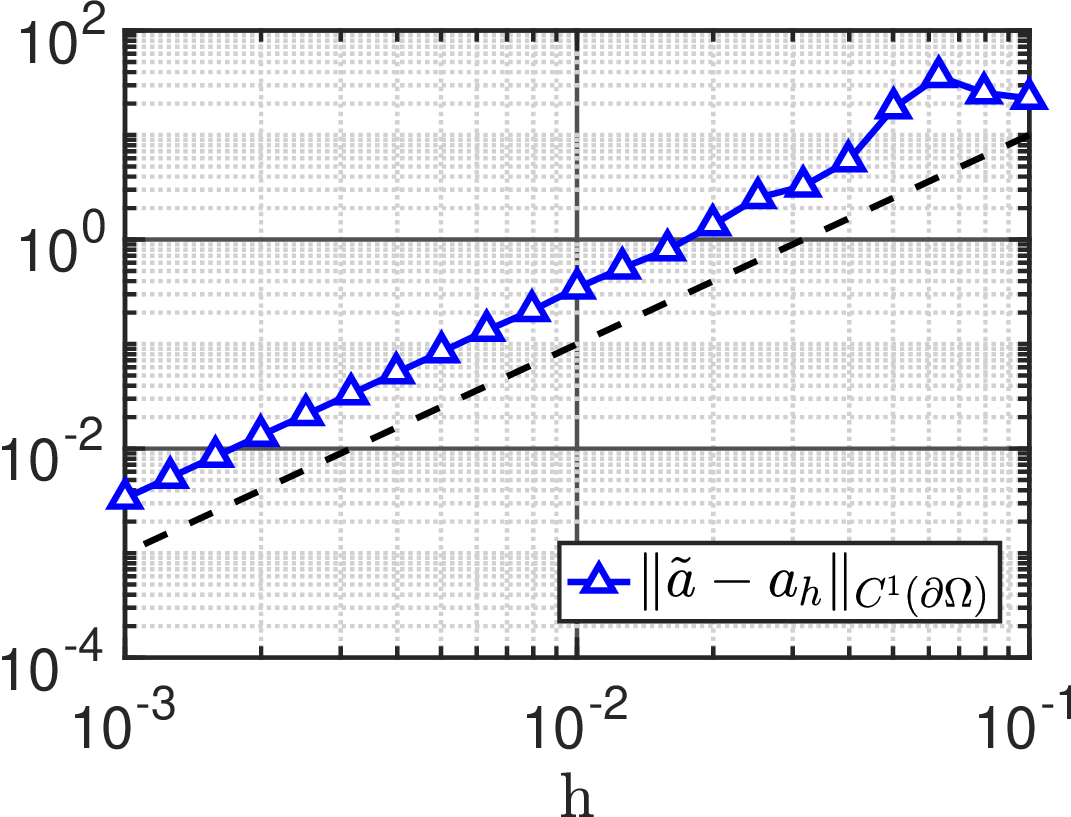}}
\end{minipage}
\hfill
\begin{minipage}[t]{0.45\textwidth}
        \adjustbox{valign=t,height=.28cm}{
\begin{tabular}{lll}
        \toprule
        $h$ & $\Vert \tilde{a} - a_h\Vert_{C^1(\partial \Omega)}$ & EOC \\
        \midrule
        \midrule
        $5.0\times 10^{-3}$    & $8.4\times 10^{-2}$     & $2.05$     \\
        $4.0\times 10^{-3}$    & $5.3\times 10^{-2}$     & $2.00$     \\
        $3.1\times 10^{-3}$    & $3.3\times 10^{-2}$    & $2.01$     \\
        $2.5\times 10^{-3}$    & $2.1\times 10^{-2}$     & $2.00$     \\
        $2.0\times 10^{-3}$    & $1.3\times 10^{-2}$     & $1.98$     \\
        $1.6\times 10^{-3}$    & $8.4\times 10^{-3}$     & $2.02$     \\
        $1.3\times 10^{-3}$   & $5.3\times 10^{-3}$     & $2.00$     \\
        $1.0\times 10^{-3}$   & $3.4\times 10^{-3}$     & $2.00$     \\
        \bottomrule
    \end{tabular}
    }
     \end{minipage}
\caption{Left: Double logarithmic plot showing the maximal mesh size $h$ against the error $\Vert \tilde{a}-a_h\Vert_{C^1(\partial \Omega)}$. The dashed line has slope 2.
Right: Some maximal mesh sizes together with their error and the EOC.}
\label{fig:ex2}
\end{figure}
We find that the predicted quadratic convergence that we expected from \eqref{eq:h2bound} is nicely attained. 

\textbf{Example 3.}
In this example we study the reconstruction of $\tilde{a}$ from noisy data $q^\delta$ given by $q^\delta = q + \delta$, where $\delta$ is some perturbation.
For some $\sigma>0$ we choose $\delta$ to be a Hadamard type function\footnote{Functions defined as $\delta(x) = \mathrm{e}^{\mathrm{i}x\cdot z}$ with $z=a+\mathrm{i}b$, $a,b \in \R^2$, $|a|=|b|$ and $a\cdot b=0$ are harmonic functions in $\Omega$.} defined as $\delta = \delta(\sigma) = \sigma(\delta_1 + \delta_2)$ with $$\delta_i(x) \, =\, \operatorname{Re}(\mathrm{e}^{\mathrm{i} (x-c_i) \cdot z})1_{B_{c_i}(x_i)}(x)\, ,$$ where $x_i$ and $c_i$ are the respective middle points and radii of the circles that $\omega$ consists of.
In our example, we choose $z=[10, \, 10i]^\top$.
We perform several simulations for a variety of mesh sizes $h$ ranging from $10^{-1}$ to $10^{-3}$ with noise levels $\sigma \in \{10^{-4},\,  10^{-5},\,  10^{-6}\}$. For each reconstruction we start with the initial guess $a=1$ and compute the error $\Vert \tilde{a} - a_h^{\delta(\sigma)}\Vert_{C^1(\partial \Omega)}$, where $a_h^{\delta(\sigma)}$ denotes the final iteration corresponding to a given mesh size $h$ and noisy data $\delta(\sigma)$.
In Figure~\ref{fig:ex3} we visualize the result of this experiment. 
\begin{figure}[t!]
\centering 
\begin{minipage}[t]{0.45\textwidth}
\adjustbox{valign=t,height=.27cm}{\includegraphics[scale=.38]{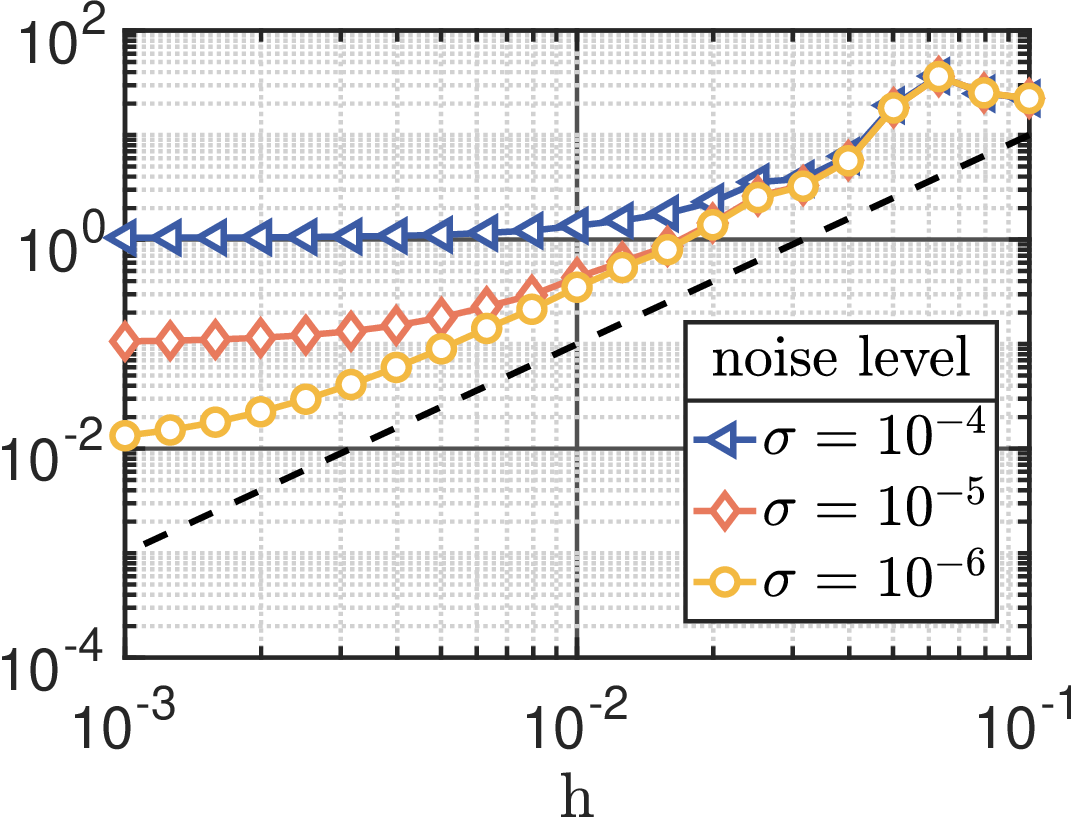}}
\end{minipage}
\hfill
\begin{minipage}[t]{0.45\textwidth}
\adjustbox{valign=t,height=.27cm}{\includegraphics[scale=.355]{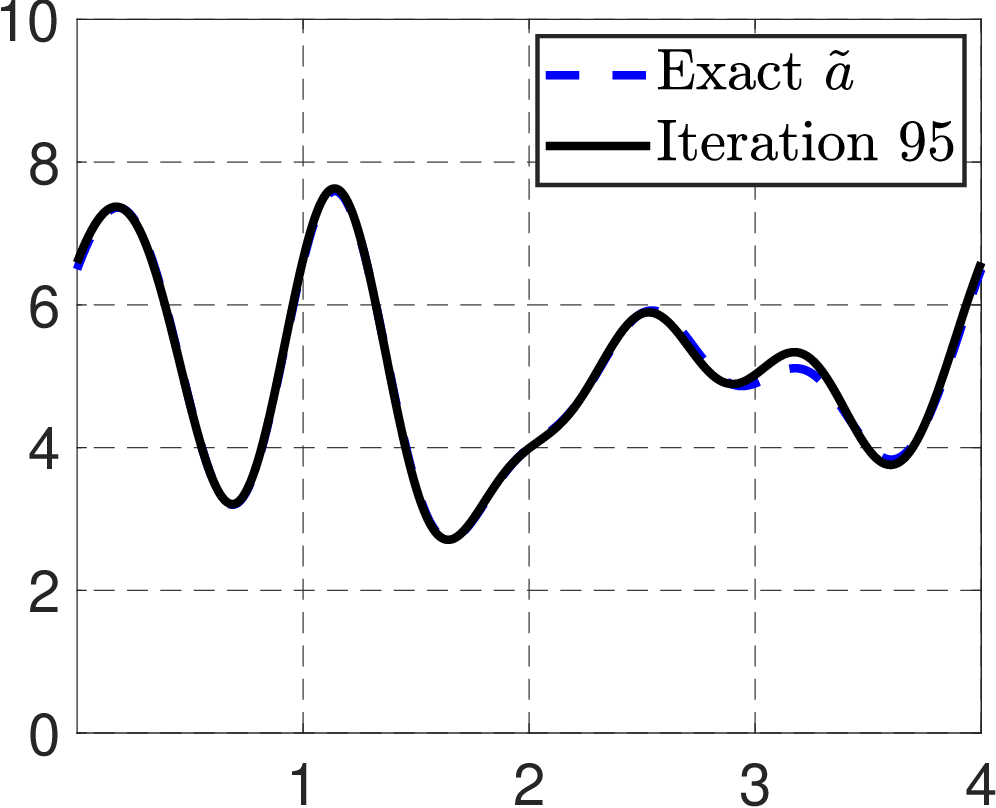}}
\end{minipage}
\caption{Left: Double logarithmic plot showing the maximal mesh size $h$ against the error $\Vert \tilde{a}-a_h^{\delta(\sigma)}\Vert_{C^1(\partial \Omega)}$ for three different noise levels $\sigma$. The dashed line has slope 2.
Right: The final iterate of the simulation corresponding to $h=10^{-3}$ and $\sigma = 10^{-4}$ together with the exact Robin parameter $\tilde{a}$.
}
\label{fig:ex3}
\end{figure}
Its left plot is double logarithmic and shows the mesh size $h$ against the error $\Vert \tilde{a} - a_h^{\delta(\sigma)}\Vert_{C^1(\partial \Omega)}$.
We find that the result is in accordance with the bound \eqref{eq:h2boundpert}.
The error decays quadratically, until the term $\Vert \delta \Vert_{L^2(\omega)}$ dominates the error bound.
In the right plot of Figure~\ref{fig:ex3} we see the final iteration for $h=10^{-3}$ and noise level $\sigma = 10^{-6}$.
Even though the error stagnates around the value $1$ for this noise level, the approximation $a_h^\delta$ visually only differs around $t=3$ (which corresponds to the top-left corner of $\Omega$) from the exact $\tilde{a}$.

\textbf{Example 4.}
For this example we consider an approximation space for the reconstruction algorithm that is smaller than the exact space $V_{J_1,J_2} $.
Since $\tilde{a}$ from \eqref{eq:atilde} lies in $V_{6,6}$, but not in any subspace $V_{j_1,j_2}$ with $j_1<J_1$ and $j_2<J_2$, we expect that the final iterate of the reconstruction algorithm would not capture the true features of the Robin parameter $\tilde{a}$.
For our experiment we let the tuple $(j_1, j_2)$ lie in the set $\{ (3,3),\, (4,4),\, (5,5) \}$ and start our reconstruction starting from the constant initial guess $a=1$.
The final iterates for each simulation can be found in Figure~\ref{fig:ex4}.
\begin{figure}[t!]
\centering 
\includegraphics[scale=.24]{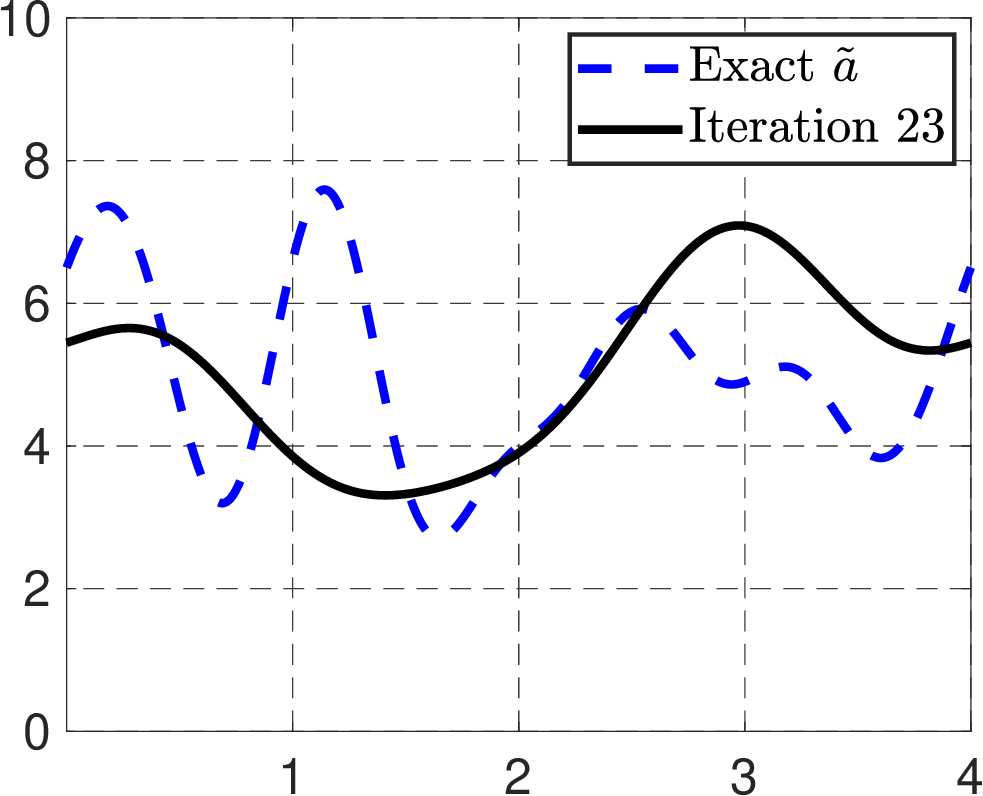}\hfill
\includegraphics[scale=.24]{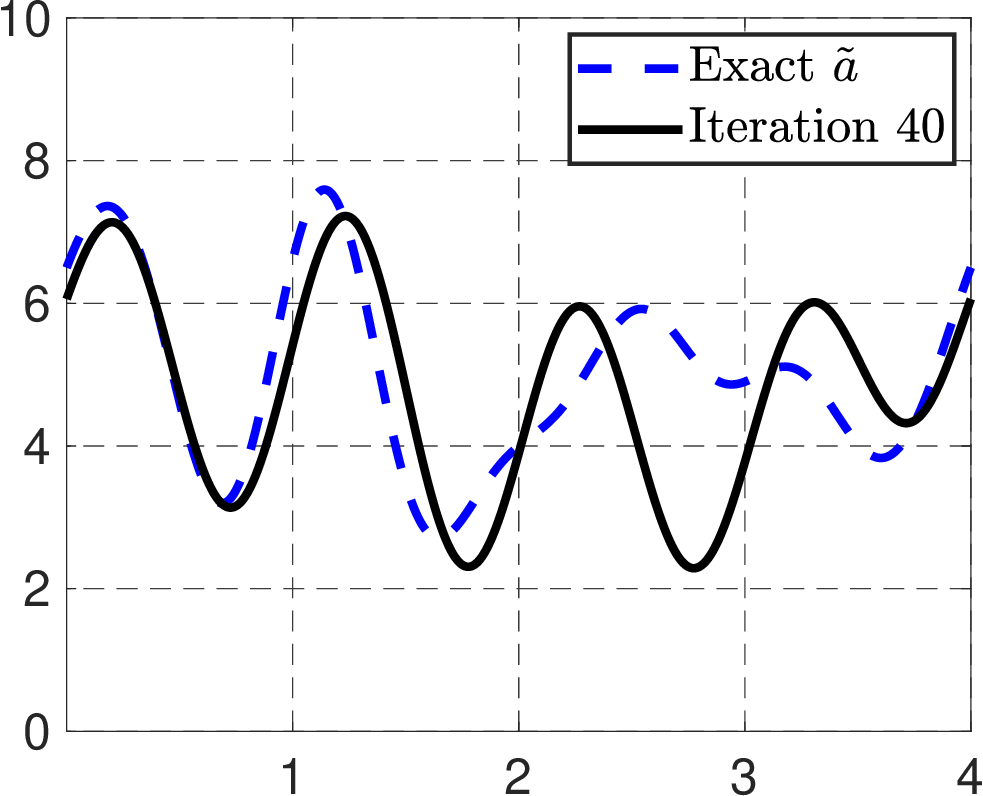}\hfill
\includegraphics[scale=.24]{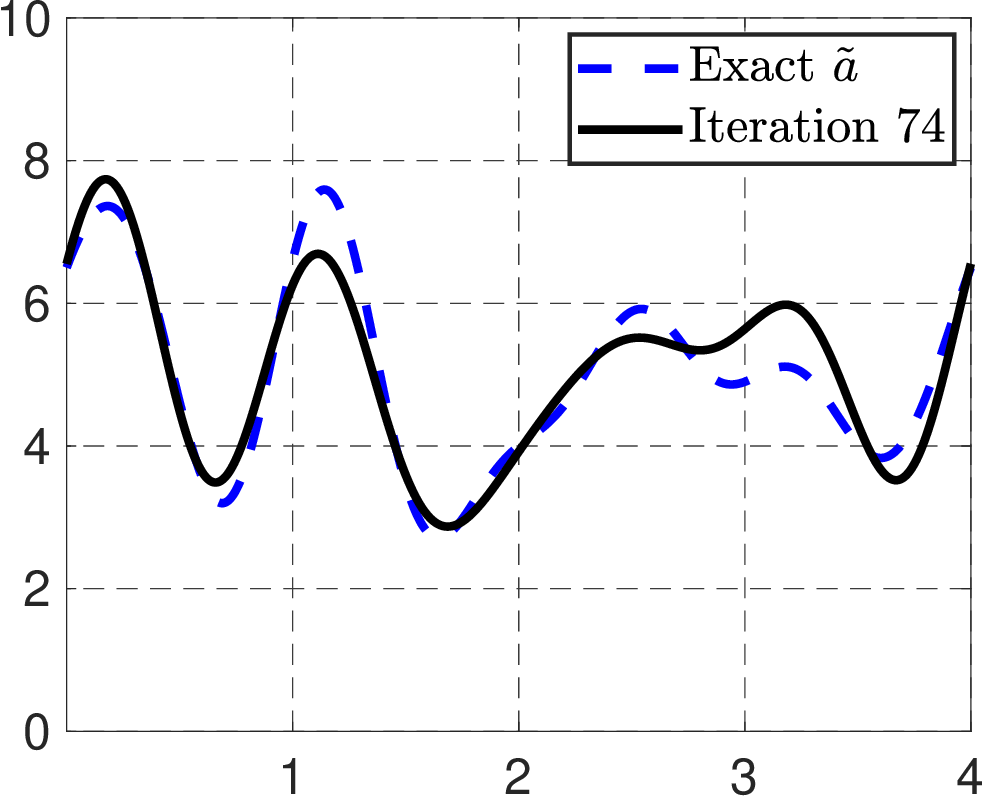}\hfill
\caption{Final iterates for smaller approximation spaces. The approximation space for the reconstruction scheme is
spanned by $\phi_{1,m-1}$ and $\phi_{2,n}$ from \eqref{eq:phifun} with $1\leq m\leq j_1$ and $1 \leq n \leq j_2$ for $(j_1,j_2)$ equal to $(3,3)$ (left), $(4,4)$ (middle) and $(5,5)$ (right).
}
\label{fig:ex4}
\end{figure}
For increasing $j_1$ and $j_2$ we see that the final iterates get closer to the true Robin parameter.
In particular, the final iterate for $(j_1,j_2)=(5,5)$ found in the right plot of Figure~\ref{fig:ex4} already shares most features of $\tilde{a}$.
We note that the approximation in the smaller approximation space does not coincide with the projection of $\tilde{a}$ onto these spaces.

\textbf{Example 5.}
In our final example we study the condition number of the matrix $J_{F_h}$ at $\tilde{a}$ from step \ref{algstep4} in Algorithm~\ref{algo:newton} numerically.
This matrix is a discrete version of the Fr\'echet derivative $\dot{F}_h[\tilde{a}]$ from \eqref{eq:dotFh}.
We consider the vectors $\alpha$ and $\beta$ from \eqref{eq:alphabeta} and prolong both with two more entries, respectively, which are all $1$.
For this example exclusively, we compute $q$ by using a first order finite element approximation and the same mesh, as for the computation of $J_{F_h}$.
The finest mesh size in this case is $h\approx 7\times 10^{-4}$.
Starting with the dimensions $J_1=J_2=2$ we successively first increase $J_1$ and then $J_2$ until $J_1=J_2=8$ is reached, i.e., until the finite dimensional subspace $V_{J_1,J_2}$ has 16 dimensions.
For each dimension from $4$ to $16$ we compute the condition number $\kappa(J_{F_h})$.
The result is found in Figure~\ref{fig:ex5}.
We see that the condition number increases exponentially as the dimension $J = J_1+J_2$ grows.
The proof of Lemma~\ref{lem:Fh} requires Corollary~\ref{cor:dotuh}, which, on the other hand, uses the stability result from Theorem~\ref{thm:stability}.
In the proof of Theorem~\ref{thm:stability} the norm equivalence on finite dimensional spaces has been used.
We assume that the norm equivalence is reflected in the exponential growth of the condition number of $J_{F_h}$ in Figure~\ref{fig:ex5}.
\begin{figure}[t!]
\centering 
\includegraphics[scale=.3]{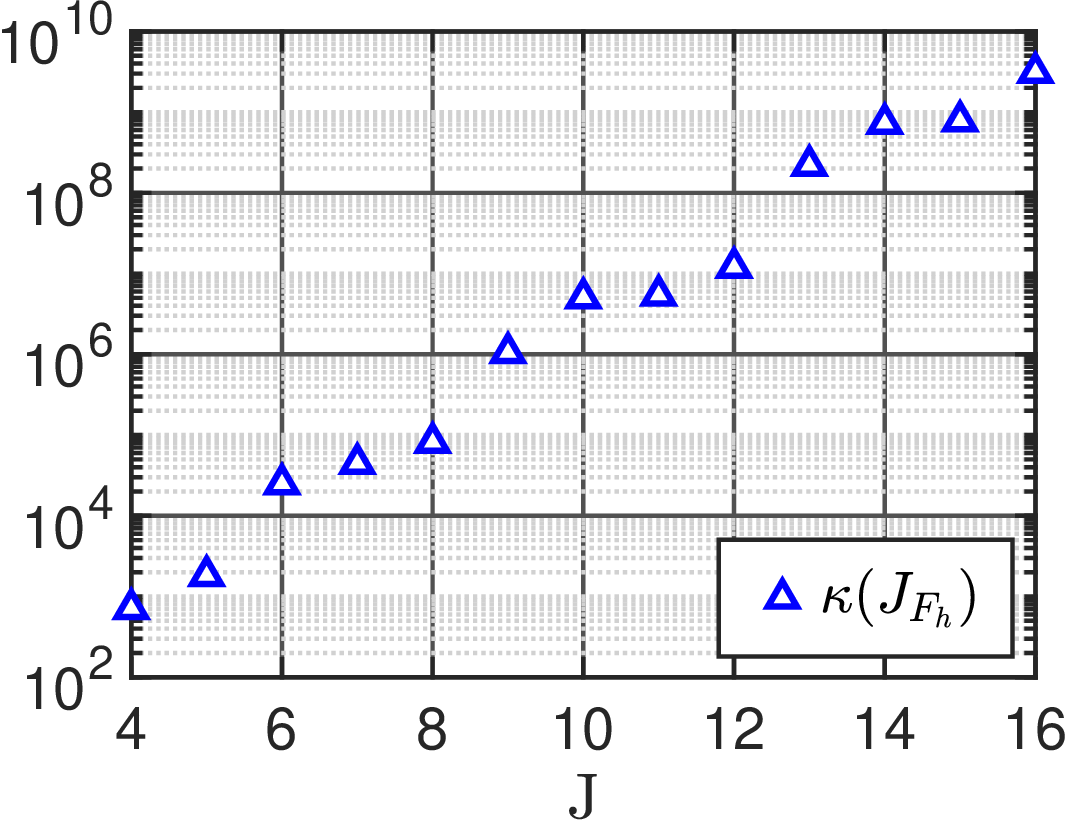}
\caption{Condition number of the Jacobian $J_{F_h}$ from step \ref{algstep4} in Algorithm~\ref{algo:newton} with respect to different finite dimensions $J=J_1+J_2$ of the subspace $V_{J_1,J_2}$.}
\label{fig:ex5}
\end{figure}

\bibliographystyle{abbrvurl}
\bibliography{references}

\end{document}